\documentclass[11pt,a4paper]{article}%
\usepackage{amssymb,amsmath,amsfonts,amsthm,array,bm,color}%
\usepackage{amsmath}%
\usepackage{amsfonts}%
\usepackage{amssymb}%
\usepackage{graphicx}
\providecommand{\U}[1]{\protect\rule{.1in}{.1in}}
\newtheorem{theorem}{Theorem}[section]
\newtheorem{lemma}[theorem]{Lemma}
\newtheorem{corollary}[theorem]{Corollary}
\newtheorem{proposition}[theorem]{Proposition}
\newtheorem{remark}[theorem]{Remark}

\def\<{\langle}
\def\>{\rangle}
\def\d{{\rm d}}
\def\L{\mathcal{L}}

\def\E{\mathbb{E}}
\def\N{\mathbb{N}}
\def\P{\mathbb{P}}
\def\R{\mathbb{R}}
\def\T{\mathbb{T}}
\def\Z{\mathbb{Z}}

\def\eps{\varepsilon}

\begin{document}

\makeatletter
\renewcommand\theequation{\thesection.\arabic{equation}}
\@addtoreset{equation}{section} \makeatother

\title{Convergence of transport noise to Ornstein--Uhlenbeck for 2D Euler equations under the enstrophy measure}

\author{Franco Flandoli\footnote{Email: franco.flandoli@sns.it. Scuola Normale Superiore of Pisa, Italy.} \ and
Dejun Luo\footnote{Email: luodj@amss.ac.cn. RCSDS, Academy of Mathematics and Systems Science, Chinese Academy of Sciences, Beijing 100190, China, and School of Mathematical Sciences, University of the Chinese Academy of Sciences, Beijing 100049, China. }}

\maketitle

\begin{abstract}
We consider the vorticity form of the 2D Euler equations which is perturbed by a suitable transport type noise and has white noise initial condition. It is shown that, under certain conditions, this equation converges to the 2D Navier--Stokes equation driven by the space-time white noise.
\end{abstract}



\section{Introduction}

Navier--Stokes equations in dimension 2 with additive space-time white noise%
  \begin{equation}\label{NS}
  \aligned
  \d u+\left(  u\cdot\nabla u+\nabla p\right)  \d t  & =\nu\Delta u\, \d t+\alpha\, \d W,\\
  \operatorname{div}u  & =0
  \endaligned
  \end{equation}
have been the object of several investigations, \cite{DaPD, Debussche, AF, Stannat, AF2, Sauer, ZhuZhu} among others and, with its first-stage renormalization,
even contributed to the development of some of the ideas around Regularity
Structures. One of the main features is the Gaussian invariant measure
formally given by
  \begin{equation}\label{enstrophy-meas}
  \mu\left( \d\omega\right)  =Z^{-1}\exp\big(  -\beta\| \omega \|_{L^{2}}^{2} \big) \, \d\omega
  \end{equation}
($\beta>0$ related to the constants of equations (\ref{NS}) and the domain)
where we have denoted by $\omega$ the vorticity associated to the velocity
field $u$ and where $\left\Vert \omega\right\Vert _{L^{2}}^{2}$ denotes the
enstrophy (hence $\mu$ is often called enstrophy measure). This equation is
well posed in suitable function spaces, even in the strong probabilistic
sense. For the purpose of the next description, it is convenient to
reformulate the equation in vorticity form%
  \begin{equation}
  \d\omega+u\cdot\nabla\omega \,\d t=\nu\Delta\omega \,\d t+ \alpha \nabla^{\perp}\cdot \d W, \label{NS vort}%
  \end{equation}
where, as said above, $\omega=\nabla^{\perp}\cdot u$ and, for a vector field
$v$, $\nabla^{\perp}\cdot v$ denotes $\partial_{2}v_{1}-\partial_{1}v_{2}$.
Here $W$ is a solenoidal vector valued cylindrical Brownian motion.

A related model is 2D Euler equations, that in vorticity form is%
\[
\partial_{t}\omega+u\cdot\nabla\omega=0
\]
with $\omega=\nabla^{\perp}\cdot u$, $\operatorname{div}u=0$. In the sense
described in \cite{AC, F1}, the enstrophy measure $\mu$ is invariant
also for this equation (for every $\beta>0$, in this case). The same fact
holds for a stochastic version of 2D Euler equations, but with transport type
noise, as described in \cite{FL-1, FL}:%
\[
\d\omega+u\cdot\nabla\omega \,\d t=\sum_{k}\sigma_{k}\cdot\nabla\omega\circ \d W^{k}, %
\]
where $\sigma_{k}(x)  $ are divergence free vector fields and
$W^{k}$ independent Brownian motions. We focus our discussion on the 2D torus
$\mathbb{T}^{2}=\mathbb{R}^{2}/\mathbb{Z}^{2}$ and choose, to fix notations,
\[
\sigma_{k}(x)  =\frac{1}{\sqrt{2}}\frac{k^{\perp}}{\left\vert
k\right\vert ^{\gamma}}e_{k}(x), \quad k\in \Z_0^2,
\]
where $\Z_0^2= \Z^2\setminus\{0\}$ and $e_{k}(x)  $ is the orthonormal basis of sine and cosine
functions, see \eqref{ONB} below. In \cite{FL-1, FL} the problem has been studied for
$\gamma>2$.

The purpose of this paper is to present a rather unexpected link between these
two subjects. Based on \cite{FL-1, FL}, it is interesting to ask what happens when
$\gamma=2$, limiting case where certain terms diverge. For instance, the
It\^{o}--Stratonovich correction of the multiplicative noise above diverges
proportionally to $\sum_{|k|\leq N}\frac{1}{\left\vert
k\right\vert ^{2}}$ as $N\rightarrow\infty$. We therefore investigate whether
this divergence may be compensated by an infinitesimal coefficient in front of
the noise:%
\begin{equation}\label{approx-eq}
\d\omega+u\cdot\nabla\omega\, \d t=2\sqrt{\nu}\,\varepsilon_{N}\sum_{|k|\leq N%
}\frac{k^{\perp}}{\left\vert k\right\vert ^{2}}e_{k}\cdot\nabla\omega\circ \d W^{k}, %
\end{equation}
where $\varepsilon_{N}= \big(\sum_{|k|\leq N} \frac1{|k|^2} \big)^{-1/2} \sim\frac{1}{\sqrt{\log N}}$.
The result, described below, is that this model, hyperbolic in nature, converges to the parabolic equation
\eqref{NS vort} above with $\alpha = \sqrt{2\nu}$, provided that $\nu$ is not too small.

Let us explain a vague physical intuition about this result, which however is
not sufficient to state a firm conjecture, without a due detailed
investigation. Transport multiplicative noise $\sum_{k}\sigma_{k}\cdot
\nabla\omega\circ \d W^{k}$ provokes a random Lagrangian displacement of ``fluid
particles''. Assume that the space-covariance of the Gaussian field $\sum
_{k}\sigma_{k}\left(  x\right)  W_{t}^{k}$ is concentrated around zero, as it
is in the scaling limit investigated in this work. Look at fluid particles as
an interacting system of particles;\ the effect of the Gaussian field on
different particles is almost independent, when the distance between particles
is not too small (see [7, Introduction] for related discussions). Thus,
approximatively, it is like driving each particle with an independent noise,
and we know from mean field theories that independent Brownian perturbation of
particles reflects into a Laplacian in the limit PDE. This intuitively
explains the presence of the Laplacian in the limit equation, but the presence
also of a white noise is less clear.

Let us also emphasize another nontrivial aspect that could be misunderstood.
Technically speaking, a Laplacian (or a more complicated second order
differential operator)\ arises when rewriting a Stratonovich multiplicative
transport noise in It\^{o}'s form (see Section 2 below). But this does not mean
that the original equation, with transport noise, is parabolic. The original
equation is hyperbolic, and the solution (when smooth enough) is the stochastic
Lagrangian transport of the initial condition. Thus it is a nontrivial fact
that a truly parabolic equation is obtained in the scaling limit investigated
in the present work.

This paper is organized as follows. In Section \ref{sec-convergence} we prove the main result (Theorem \ref{thm-convergence}) which states that, under suitable conditions, the white noise solutions of a sequence of stochastic Euler equations converge weakly to the solution of the Navier--Stokes equation driven by space-time white noise. We solve in Section \ref{sec-Kolmogorov-eq} the corresponding Kolmogorov equation by using the Galerkin approximation. Finally, in the first part of Section 4 we recall a decomposition formula which plays an important role in the proof, and in Section \ref{sec-coincidence-nonlinear-part} we prove the coincidence of two different definitions of the nonlinear part in the Euler equation.

\section{Convergence of the equations \eqref{approx-eq}}\label{sec-convergence}

First, we introduce some notations. We denote by
  \begin{equation}\label{ONB}
  e_k(x)= \sqrt{2} \begin{cases}
  \cos(2\pi k\cdot x), & k\in \Z^2_+, \\
  \sin(2\pi k\cdot x), & k\in \Z^2_-,
  \end{cases} \quad x\in \T^2,
  \end{equation}
where $\Z^2_+ = \big\{k\in \Z^2_0: (k_1 >0) \mbox{ or } (k_1=0,\, k_2>0) \big\}$ and $\Z^2_- = -\Z^2_+$. Then $\{e_k: k\in \Z_0^2\}$ constitute a CONS of $L^2_0(\T^2)$, the space of square integrable functions with zero mean. Define
  \begin{equation}\label{vector-fields}
  \sigma_k(x)= \frac1{\sqrt{2} } \frac{k^\perp}{|k|^2} e_k(x), \quad k\in \Z^2_0,
  \end{equation}
with $k^\perp = (k_2,-k_1)$. Let $\nu>0$ be fixed and, for $N\geq 1$, define $\Lambda_N= \{k\in \Z_0^2: |k| \leq N\}$. We rewrite the equation \eqref{approx-eq} as
  \begin{equation}\label{vorticity-Euler-1}
  \d \omega^N_t + u^N_t\cdot \nabla\omega^N_t = 2\sqrt{2\nu}\, \eps_N \sum_{k\in \Lambda_N} \sigma_k\cdot \nabla\omega^N_t \circ\d W^k_t.
  \end{equation}
Here $\omega^N_t = \nabla^\perp \cdot u^N_t$ and conversely, $u^N_t$ is represented by $\omega^N_t$ via the Biot--Savart law:
  $$u^N_t(x)= \big(\omega^N_t \ast K\big)(x) = \big\<\omega^N_t, K(x- \cdot ) \big\>,$$
with $K$ the Biot--Savart kernel on $\T^2$:
  $$K(x)= 2\pi {\rm i} \sum_{k\in \Z^2_0} \frac{k^\perp}{|k|^2} {\rm e}^{2\pi {\rm i} k \cdot x} = -2\pi \sum_{k\in \Z^2_0} \frac{k^\perp}{|k|^2} \sin(2\pi k \cdot x).$$

We assume that the initial data $\omega^N_0$ of \eqref{vorticity-Euler-1} is a white noise on $\T^2$; namely, $\omega^N_0$ is a random variable defined on some probability space $(\Theta, \mathcal F, \P)$, taking values in the space of distributions $C^\infty(\T^2)'$ on $\T^2$, such that, for any $\phi\in C^\infty(\T^2)$, $\big\<\omega^N_0, \phi\big>$ is a centered Gaussian random variable with variance $\|\phi\|_{L^2(\T^2)}^2$. From the definition, we easily deduce that
  $$\E \big\<\omega^N_0, \phi\big> \big\<\omega^N_0, \psi\big> = \<\phi, \psi\>_{L^2(\T^2)} \quad \mbox{for any } \phi, \psi \in C^\infty(\T^2). $$
We denote the law of $\omega^N_0$ by $\mu$, which is also called the enstrophy measure with the heuristic expression  \eqref{enstrophy-meas}. It is not difficult to show that $\mu$ is supported by $H^{-1-}(\T^2) = \cap_{s>0} H^{-1-s}(\T^2)$, where, for any $r\in \R$, $H^r(\T^2)$ is the usual Sobolev space on $\T^2$.

For any fixed $N\geq 1$, following the proof of \cite[Theorem 1.3]{FL-1}, we can show that the equation \eqref{vorticity-Euler-1} has a white noise solution $\omega^N\in C\big([0,T], H^{-1-}(\T^2)\big)$ (possibly defined on a new probability space); namely, for any $t\in [0,T]$, $\omega^N_t$ is distributed as the white noise measure $\mu$, and for any $\phi\in C^\infty(\T^2)$,
  \begin{equation}\label{white-noise-solu}
  \aligned
  \big\<\omega^N_t, \phi\big\> &= \big\<\omega^N_0, \phi\big\> +\int_0^t \big\<\omega^N_r\otimes \omega^N_r, H_\phi\big\>\, \d r - 2\sqrt{2\nu}\, \varepsilon_N \sum_{k\in \Lambda_N} \int_0^t \big\<\omega^N_r,\sigma_k \cdot \nabla \phi \big\>\,\d W^k_r\\
  &\hskip13pt + 4\nu \varepsilon_N^2 \sum_{k\in \Lambda_N} \int_0^t \big\<\omega^N_r,\sigma_k \cdot \nabla (\sigma_k \cdot \nabla \phi) \big\>\,\d r.
  \endaligned
  \end{equation}
Moreover, it is easy to show that $\omega^N$ is a stationary process, which is a consequence of the same result for the stochastic point vortex dynamics proved in \cite[Proposition 2.3]{FL-1}. Our purpose is to show that, if $\nu$ is not too small, the equations \eqref{vorticity-Euler-1} converge in some sense to
  \begin{equation}\label{vorticity-NSE}
  \d \omega_t + u_t\cdot \nabla \omega_t \,\d t = \nu \Delta \omega_t \,\d t+ \sqrt{2\nu}\, \nabla^\perp \cdot \d W_t, \quad \omega_0\stackrel{d}{\sim} \mbox{white noise on } \T^2.
  \end{equation}

\begin{remark}\label{sec-2-remark}
Some explanations for the nonlinear term in \eqref{white-noise-solu} are necessary. For $\phi \in C^\infty(\T^2)$,
  $$H_\phi(x,y) := \frac12 K(x-y)\cdot (\nabla\phi(x) -\nabla\phi(y)), \quad x,y\in \T^2,$$
with $K$ the Biot--Savart kernel and the convention that $H_\phi(x,x)=0$. It is well known that, for all $x\in \T^2 \setminus \{0\}$, $K(-x) =-K(x)$ and $|K(x)|\leq C/|x|$ for some constant $C>0$; thus $H_\phi$ is symmetric and
  \begin{equation}\label{non-linear-drift}
  \|H_\phi\|_\infty \leq C \|\nabla^2\phi\|_\infty.
  \end{equation}
Since $\omega^N_r$ is a white noise on $\T^2$ for any $r\in [0,T]$, the quantity $\big\<\omega^N_r\otimes \omega^N_r, H_\phi\big\>$ is well defined as a limit in $L^2(\Theta,\P)$ of an approximating sequence, see \cite[Theorem 8]{F1} for details. According to the arguments in Section \ref{sec-coincidence-nonlinear-part}, this definition is consistent with that defined by the Galerkin approximation; the latter will be used in Section \ref{sec-Kolmogorov-eq}.
\end{remark}

First we follow the arguments in \cite[Section 3]{FL-1} to show that the family of distributions $\big\{Q^N \big\}_{N\geq 1}$ of $\omega^N$ on $\mathcal X:= C\big([0,T], H^{-1-}(\T^2) \big)$ is tight. To this end, we need to apply the compactness criterion proved in \cite[p. 90, Corollary 9]{Simon}. We state it here in our context.

Take $\delta\in (0,1)$ and $\kappa>5$ (this choice is due to estimates below) and consider the spaces
  $$X=H^{-1-\delta/2}(\T^2),\quad B=H^{-1-\delta}(\T^2),\quad Y=H^{-\kappa}(\T^2).$$
Then $X\subset B\subset Y$ with compact embeddings and we also have, for a suitable constant $C>0$ and for
  \begin{equation}\label{eq-theta}
  \theta= \frac{\delta/2}{\kappa -1-\delta/2},
  \end{equation}
the interpolation inequality
  $$\|\omega\|_B \leq C \|\omega\|_X^{1-\theta} \|\omega\|_Y^\theta,\quad \omega\in X.$$
These are the preliminary assumptions of \cite[p. 90, Corollary 9]{Simon}. We consider here a particular case:
  $$\mathcal S= L^{p_0}(0,T; X)\cap W^{1/3,4}(0,T; Y),$$
where for $0< \alpha <1$ and $p\geq 1$,
  $$W^{\alpha,p}(0,T; Y)=\bigg\{f: \, f\in L^p(0,T; Y) \mbox{ and } \int_0^T\! \int_0^T \frac{\|f(t)-f(s)\|_Y^p}{|t-s|^{\alpha p+1}}\,\d t\d s <\infty\bigg\}.$$
The next result is taken from \cite[Lemma 3.1]{FL-1}.

\begin{lemma}\label{lem-embedding}
Let $\delta\in (0,1)$ and $\kappa>5$ be given. If
  $$p_0> \frac{12(\kappa -1-3\delta/2)}\delta,$$
then $\mathcal S$ is compactly embedded into $C\big([0,T], H^{-1-\delta}(\T^2) \big)$.
\end{lemma}

\begin{proof}
Recall that $\theta$ is defined in \eqref{eq-theta}. In our case, we have $s_0=0, r_0=p_0$ and $s_1=1/3, r_1=4$. Hence $s_\theta = (1-\theta)s_0 +\theta s_1= \theta/3$ and
  $$\frac1{r_\theta} = \frac{1-\theta}{r_0} + \frac\theta{r_1} = \frac{1-\theta}{p_0} + \frac\theta 4.$$
It is clear that for $p_0$ given above, it holds $s_\theta> 1/r_\theta$, thus the desired result follows from the second assertion of \cite[Corollary 9]{Simon}.
\end{proof}

Next, since $H^{-1-}(\T^2)$ is endowed with the Fr\'{e}chet topology, one can prove

\begin{lemma}\label{lem-tight}
The family $\big\{Q^N\big\}_{N\geq 1}$ is tight in $\mathcal X$ if and only if it is tight in $C\big([0,T], H^{-1-\delta}(\T^2) \big)$ for any $\delta>0$.
\end{lemma}

The proof is similar to Step 1 of the proof of \cite[Proposition 2.2]{FL-1} and we omit it here. In view of the above two lemmas, it is sufficient to prove that  $\big\{Q^N\big\}_{N\geq 1}$ is bounded in probability in $W^{1/3,4}\big(0,T; H^{-\kappa}(\T^2) \big)$ and in each $L^p\big(0,T; H^{-1-\delta}(\T^2) \big)$ for any $p>0$ and $\delta>0$.

Before moving further, we recall some properties of the white noise which will be frequently used below.

\begin{lemma}\label{properties-WN}
Let $\xi:(\Theta, \mathcal F, \P)\to C^\infty(\T^2)'$ be a white noise on $\T^2$. Then for any $p>1$ and $\delta>0$, there exist $C_p>0, \,C_{p,\delta}>0$ such that
\begin{itemize}
\item[\rm (1)]  $\E \big(|\<\xi, \phi\>|^p \big) \leq C_p\|\phi\|_\infty^p$ for all $\phi\in C^\infty(\T^2)$;
\item[\rm (2)] $\E\big( \|\xi \|_{H^{-1-\delta}}^{p} \big) \leq C_{p,\delta}$;
\item[\rm (3)] $\E\big(|\<\xi\otimes \xi, H_\phi\>|^p\big)\leq C_p \|\nabla^2\phi \|_\infty^p$ for all $\phi \in C^\infty(\T^2)$.
\end{itemize}
\end{lemma}

\begin{proof}
The first assertion follows from the fact that $\<\xi, \phi\>$ is a centered Gaussian random variable with variance $\|\phi\|_{L^2(\T^2)}^2$. Applying this result to $\phi =e_k$, we can deduce the second estimate from the  definition of the Sobolev norm  $\|\cdot \|_{H^{-1-\delta}}$.

We turn to prove the last one. Let $H^n_\phi,\, n\geq 1$ be the smooth approximations of $H_\phi$ constructed in \cite[Remark 9]{F1}, satisfying
  $$\|H^n_\phi\|_\infty \leq \|H_\phi\|_\infty \leq C \|\nabla^2\phi \|_\infty,$$
where the last inequality is due to \eqref{non-linear-drift}. By \cite[Corollary 6(i)]{F1}, we have
  $$\E\big(|\<\xi\otimes \xi, H^n_\phi\>|^p\big)\leq C_p \|H^n_\phi\|_\infty^p \leq C'_p \|\nabla^2\phi \|_\infty^p.$$
This implies the family $\{\<\xi\otimes \xi, H^n_\phi\>\}_{n\geq 1}$ is bounded in any $L^p(\Theta, \P),\, p>1$, which, combined with the fact that $\<\xi\otimes \xi, H^n_\phi\>$ converges to $\<\xi\otimes \xi, H_\phi\>$ in $L^2(\Theta, \P)$ (see \cite[Theorem 8]{F1}), yields the desired result.
\end{proof}

We first note that, for any $p>1$ and $\delta>0$, by (2) of Lemma \ref{properties-WN},
  \begin{equation}\label{sec-3.1}
  \E\bigg[\int_0^T \big\|\omega^N_t \big\|_{H^{-1-\delta}}^{p} \,\d t\bigg] = \int_0^T \E\big[ \big\|\omega^N_t \big\|_{H^{-1-\delta}}^{p} \big]\,\d t \leq C_{p, \delta} T,\quad \mbox{for all } N\geq 1.
  \end{equation}
Next, similar to \cite[Lemma 3.3]{FL-1}, we can prove

\begin{lemma}\label{lem-estimate}
There exists $C>0$ such that for any $\phi\in C^\infty(\T^2)$, we have
  $$\E\big[ \big\<\omega^N_t- \omega^N_s, \phi \big\>^4 \big]\leq C (t-s)^2\big( \|\nabla \phi\|_\infty^4 + \|\nabla^2 \phi\|_\infty^4 \big).$$
\end{lemma}

\begin{proof}
The proof is almost the same as that of \cite[Lemma 3.3]{FL-1}. By \eqref{white-noise-solu}, we have
  \begin{equation}\label{lem-estimate-1}
  \aligned
  \big\<\omega^N_t- \omega^N_s, \phi\big\> &= \int_s^t \big\<\omega^N_r\otimes \omega^N_r, H_\phi\big\>\, \d r - 2\sqrt{2\nu}\, \varepsilon_N \sum_{k\in \Lambda_N} \int_s^t \big\<\omega^N_r,\sigma_k \cdot \nabla \phi \big\>\,\d W^k_r\\
  &\hskip13pt + 4\nu \varepsilon_N^2 \sum_{k\in \Lambda_N} \int_s^t \big\<\omega^N_r,\sigma_k \cdot \nabla (\sigma_k \cdot \nabla \phi) \big\>\,\d r.
  \endaligned
  \end{equation}
First, H\"older's inequality leads to
  \begin{equation}\label{lem-estimate-2}
  \aligned
  \E\bigg[\bigg(\int_s^t \big\<\omega^N_r\otimes \omega^N_r, H_\phi\big\>\, \d r\bigg)^{\! 4} \bigg]
  &\leq (t-s)^3 \E\bigg[\int_s^t \big\<\omega^N_r\otimes \omega^N_r, H_\phi \big\>^4\, \d r\bigg]\\
  &\leq (t-s)^3 \int_s^t C\|\nabla^2 \phi\|_\infty^4 \,\d r = C (t-s)^4  \|\nabla^2 \phi\|_\infty^4,
  \endaligned
  \end{equation}
where in the second step we used the fact that $\omega^N_r$ is a white noise and Lemma \ref{properties-WN}(3).

Next, by Burkholder's inequality,
  \begin{equation*}
  \aligned
  \E\bigg[\bigg(\varepsilon_N \sum_{k\in \Lambda_N} \int_s^t \big\<\omega^N_r,\sigma_k \cdot \nabla \phi \big\>\,\d W^k_r\bigg)^{\! 4} \bigg]
  &\leq C\varepsilon_N^4 \E \bigg[\bigg(\int_s^t \sum_{k\in \Lambda_N} \big\<\omega^N_r,\sigma_k \cdot \nabla \phi\big\>^2\,\d r\bigg)^{\! 2} \bigg]\\
  &\leq C\varepsilon_N^4 (t-s) \int_s^t \E \bigg[\bigg(\sum_{k\in \Lambda_N} \big\<\omega^N_r,\sigma_k \cdot \nabla \phi \big\>^2 \bigg)^{\! 2} \bigg] \d r.
  \endaligned
  \end{equation*}
We have by Cauchy's inequality and Lemma \ref{properties-WN}(1) that
  $$  \aligned
  \E \bigg[\bigg(\sum_{k\in \Lambda_N} \big\<\omega^N_r,\sigma_k \cdot \nabla \phi \big\>^2 \bigg)^{\! 2} \bigg]
  &= \sum_{k,l \in \Lambda_N} \E \big[ \big\<\omega^N_r,\sigma_k \cdot \nabla \phi \big\>^2 \big\<\omega^N_r,\sigma_l \cdot \nabla \phi \big\>^2 \big] \\
  &\leq \sum_{k,l \in \Lambda_N} \big[\E \big\<\omega^N_r,\sigma_k \cdot \nabla \phi\big\>^4\big]^{1/2} \big[\E \big\<\omega^N_r,\sigma_l \cdot \nabla \phi\big\>^4\big]^{1/2}\\
  &\leq C\bigg(\sum_{k\in \Lambda_N} \|\sigma_k \cdot \nabla \phi\|_\infty^2\bigg)^{\! 2} \leq \tilde C \|\nabla \phi\|_\infty^4 \bigg(\sum_{k\in \Lambda_N} \|\sigma_k\|_\infty^2\bigg)^{\! 2}.
  \endaligned $$
Note that, by \eqref{vector-fields},
  $$\sum_{k\in \Lambda_N} \|\sigma_k\|_\infty^2 = \sum_{k\in \Lambda_N} \frac1{|k|^2} = \varepsilon_N^{-2},$$
hence,
  $$\E \bigg[\bigg(\sum_{k\in \Lambda_N} \big\<\omega^N_r,\sigma_k \cdot \nabla \phi \big\>^2 \bigg)^{\! 2} \bigg] \leq C \|\nabla \phi\|_\infty^4\, \varepsilon_N^{-4}.$$
This implies
  \begin{equation}\label{lem-estimate-3}
  \E\bigg[\bigg(\varepsilon_N \sum_{k\in \Lambda_N} \int_s^t \big\<\omega^N_r,\sigma_k \cdot \nabla \phi \big\>\,\d W^k_r\bigg)^{\! 4} \bigg] \leq C(t-s)^2 \|\nabla \phi\|_\infty^4.
  \end{equation}

Finally, by H\"older's inequality,
  $$\aligned
  &\hskip13pt \E\bigg[\bigg(\varepsilon_N^2 \sum_{k\in \Lambda_N} \int_s^t \big\<\omega^N_r,\sigma_k \cdot \nabla (\sigma_k \cdot \nabla \phi) \big\> \,\d r\bigg)^{\! 4} \bigg] \\
  &\leq \varepsilon_N^8 (t-s)^3 \int_s^t \E \bigg[\bigg(\sum_{k\in \Lambda_N} \big\<\omega^N_r, \sigma_k \cdot \nabla (\sigma_k \cdot \nabla \phi) \big\>\bigg)^{\! 4} \bigg] \d r.
  \endaligned$$
Since $\sigma_k \cdot \nabla \sigma_k \equiv 0$, we have $\sigma_k \cdot \nabla (\sigma_k \cdot \nabla \phi) = \mbox{Tr}\big[ (\sigma_k \otimes \sigma_k) \nabla^2\phi \big]$. Therefore, by Lemma \ref{lem-identity} below,
  $$\sum_{k\in \Lambda_N} \sigma_k \cdot \nabla (\sigma_k \cdot \nabla \phi) = \frac14 \varepsilon_N^{-2} \Delta \phi.$$
As a result,
  $$\aligned
  \E\bigg[\bigg(\varepsilon_N^2 \sum_{k\in \Lambda_N} \int_s^t \big\<\omega^N_r,\sigma_k \cdot \nabla (\sigma_k \cdot \nabla \phi) \big\> \,\d r\bigg)^{\! 4} \bigg]
  &\leq C (t-s)^3 \int_s^t \E \big[\big\<\omega^N_r, \Delta \phi \big\>^4 \big] \d r \\
  & \leq C (t-s)^4 \|\Delta \phi\|_\infty^4.
  \endaligned$$
Combining this estimate together with \eqref{lem-estimate-1}--\eqref{lem-estimate-3}, we obtain the desired estimate.
\end{proof}

\begin{lemma}\label{lem-identity}
It holds that
  $$\sum_{k\in \Lambda_N} \sigma_k \otimes \sigma_k = \frac14 \varepsilon_N^{-2} I_2,$$
where $I_2$ is the two dimensional identity matrix.
\end{lemma}

\begin{proof}
We have
  $$\aligned
  Q_N(x) &:= \sum_{k\in \Lambda_N} \sigma_k(x)\otimes \sigma_k(x) = \sum_{k\in \Lambda_N \cap \Z^2_+} \frac{k^\perp \otimes k^\perp} {|k|^4} \big[\cos^2(2\pi k\cdot x) + \sin^2(2\pi k\cdot x)\big] \\
  &= \sum_{k\in \Lambda_N \cap \Z^2_+} \frac{1} {|k|^4} \begin{pmatrix}
  k_2^2 &  -k_1 k_2 \\  -k_1 k_2 & k_1^2
  \end{pmatrix}
  = \frac12 \sum_{k\in \Lambda_N} \frac{1} {|k|^4} \begin{pmatrix}
  k_2^2 &  -k_1 k_2 \\  -k_1 k_2 & k_1^2
  \end{pmatrix}.
  \endaligned $$
So $Q_N$ is independent on $x$. First, we have
  $$Q_N^{1,2}= - \frac12 \sum_{k\in \Lambda_N} \frac{k_1 k_2} {|k|^4} =0$$
since we can sum the four terms involving $(k_1, k_2),\, (-k_1, k_2),\, (k_1, -k_2),\, (-k_1, -k_2)$ at one time. Next,
  $$Q_N^{1,1}=  \frac12 \sum_{k\in \Lambda_N} \frac{k_2^2} {|k|^4} = \frac12 \sum_{k\in \Lambda_N} \frac{k_1^2} {|k|^4} = Q_N^{2,2}$$
since the points $(k_1, k_2)$ and $(k_2, k_1)$ appear in pair. Therefore,
  $$Q_N^{1,1}= Q_N^{2,2} = \frac14 \sum_{k\in \Lambda_N} \frac{k_1^2 + k_2^2 } {|k|^4} = \frac14 \sum_{k\in \Lambda_N} \frac{1} {|k|^2} = \frac14 \varepsilon_N^{-2}.$$
The proof is complete.
\end{proof}

Applying Lemma \ref{lem-estimate} with $\phi(x)= e_k(x)$ leads to
  $$\E\big[ \big| \big\<\omega^N_t- \omega^N_s, e_k \big\> \big|^4 \big] \leq C (t-s)^2 |k|^8, \quad k\in \Z^2_0 .$$
As a result, by Cauchy's inequality,
  $$\aligned
  \E \big( \big\|\omega^N_t- \omega^N_s \big\|_{H^{-\kappa}}^4 \big) &= \E\bigg[\bigg( \sum_k \big(1+|k|^2 \big)^{-\kappa} \big|\big\<\omega^N_t- \omega^N_s, e_k \big\> \big|^2 \bigg)^{\! 2} \bigg]\\
  &\leq \bigg(\sum_k \big(1+|k|^2 \big)^{-\kappa}\bigg) \sum_k \big(1+|k|^2 \big)^{-\kappa} \E \big[ \big|\big\<\omega^N_t- \omega^N_s, e_k \big\> \big|^4 \big]\\
  &\leq \tilde C (t-s)^2\sum_k \big(1+|k|^2 \big)^{-\kappa} |k|^8 \leq \hat C (t-s)^2,
  \endaligned$$
since $2\kappa -8 >2$ due to the choice of $\kappa$. Consequently,
  $$\E \bigg[\int_0^T\! \int_0^T \frac{ \big\|\omega^N_t- \omega^N_s \big\|_{H^{-\kappa}}^4} {|t-s|^{7/3}}\,\d t\d s\bigg] \leq \hat C \int_0^T\! \int_0^T \frac{|t-s|^2} {|t-s|^{7/3}}\,\d t\d s <\infty.$$
The proof of the boundedness in probability of $\big\{Q^N\big\}_{N\geq 1}$ in $W^{1/3,4}\big(0,T; H^{-\kappa}(\T^2) \big)$ is complete.

Combining this result with \eqref{sec-3.1} and the discussions below Lemma \ref{lem-tight}, we conclude that $\{Q^N\}_{N\geq 1}$ is tight in $\mathcal{X} = C\big([0,T], H^{-1-}(\T^2) \big)$.

Since we are dealing with the SDEs \eqref{vorticity-Euler-1}, we need to consider $Q^N$ together with the distribution of Brownian motions. Although we use only finitely many Brownian motions in \eqref{vorticity-Euler-1}, here we consider for simplicity the whole family $\big\{ (W^k_t)_{0\leq t\leq T}: k\in \Z_0^2 \big\}$. To this end, we assume $\R^{\Z_0^2}$ is endowed with the metric
  $$d_{\Z_0^2}(a,b)= \sum_{k\in \Z_0^2} \frac{|a_k-b_k| \wedge 1}{2^{|k|}}, \quad a,b \in \R^{\Z_0^2}.$$
Then $\big( \R^{\Z_0^2}, d_{\Z_0^2} \big)$ is separable and complete (see \cite[p. 9, Example 1.2]{Billingsley}). The distance in $\mathcal Y:= C\big([0,T], \R^{\Z_0^2} \big)$ is given by
  $$d_{\mathcal Y}(w,\hat w) = \sup_{t\in [0,T]} d_{\Z_0^2}(w(t), \hat w(t)),\quad w, \hat w \in \mathcal Y,$$
which makes $\mathcal Y$ a Polish space. Denote by $\mathcal W$ the law on $\mathcal Y$ of the sequence of independent Brownian motions $\big\{ (W^k_t)_{0\leq t\leq T}: k\in \Z_0^2\big\}$.

To simplify the notations, we write $W_\cdot= (W_t)_{0\leq t\leq T}$ for the whole sequence of processes $\big\{ (W^k_t)_{0\leq t\leq T}: k\in \Z_0^2 \big\}$ in $\mathcal Y$. Denote by $P^N$ the joint law of $\big(\omega^N_\cdot, W_\cdot \big)$ on $\mathcal X \times \mathcal Y,\, N\geq 1$. Since the marginal laws $\big\{ Q^N \big\}_{N\in \N}$ and $\{\mathcal W\}$ are respectively tight on $\mathcal X$ and $\mathcal Y$, we conclude that $\big\{ P^N \big\}_{N\in \N}$ is tight on $\mathcal X \times \mathcal Y$. By Skorokhod's representation theorem, there exist a subsequence $\{N_i \}_{i\in \N}$ of integers, a probability space $\big(\tilde \Theta, \tilde{\mathcal F}, \tilde \P \big)$ and stochastic processes $\big(\tilde \omega^{N_i}_\cdot, \tilde W^{N_i}_\cdot\big)$ on this space with the corresponding laws $P^{N_i}$, and converging $\tilde\P$-a.s. in $\mathcal X\times \mathcal Y$ to a limit $\big(\tilde\omega_\cdot, \tilde W_\cdot \big)$. We are going to prove that $\tilde\omega_\cdot$ solves equation \eqref{vorticity-NSE} with a suitable cylindrical Brownian motion.

First, we have the following simple result.

\begin{lemma}\label{lem-absolute-continuity}
The process $\tilde\omega_\cdot$ is stationary and for every $t\in [0,T]$, the law $\mu_t$ of $\tilde\omega_t$ on $H^{-1-}(\T^2)$ is the white noise measure $\mu$.
\end{lemma}

\begin{proof}
Recall that, for every $i\geq 1$, $\tilde\omega^{N_i}_\cdot$ has the same law as the stationary process $\omega^{N_i}_\cdot$ which solves \eqref{vorticity-Euler-1} with $N=N_i$, and has white noise measure $\mu$ as their marginal distributions. For every $m\geq 1$ and $F\in C_b\big((H^{-1-}(\T^2) )^m\big)$, $0\leq t_1<\cdots <t_m\leq T$ and $h>0$ such that $t_m+h \leq T$, since $\tilde \omega^{N_i}_\cdot$ converges to $\tilde \omega_\cdot$ a.s. in $C\big( [0,T], H^{-1-}(\T^2) \big)$, one has
  $$\aligned \tilde \E \big[F(\tilde \omega_{t_1},\cdots , \tilde \omega_{t_m})\big] &= \lim_{i\to \infty} \tilde \E \big[F\big(\tilde \omega^{N_i}_{t_1},\cdots , \tilde \omega^{N_i}_{t_m} \big)\big] = \lim_{i\to \infty} \tilde \E \big[F\big(\tilde \omega^{N_i}_{t_1+h},\cdots , \tilde \omega^{N_i}_{t_m+h} \big)\big] \\
  &= \tilde \E \big[F\big(\tilde \omega_{t_1+h},\cdots , \tilde \omega_{t_m+h} \big)\big],
  \endaligned$$
where $\tilde\E$ is the expectation on $\big(\tilde \Theta, \tilde{\mathcal F}, \tilde \P \big)$. Hence $\tilde\omega_\cdot$ is stationary. Similarly, for any $F\in C_b\big( H^{-1-}(\T^2) \big)$,
  \[ \int F(\omega)\,\d\mu_t(\omega) = \tilde\E \big[ F(\tilde \omega_t)\big]= \lim_{i\to \infty} \tilde\E \big[F\big( \tilde \omega^{N_i}_t\big)\big] = \int F(\omega)\,\d\mu(\omega). \qedhere \]
\end{proof}

Next, we show that $\big(\tilde \omega^{N_i}_\cdot, \tilde W^{N_i}_\cdot\big)$ satisfies an equation similar to that for $\big(\omega^{N_i}_\cdot, W_\cdot \big)$. By \eqref{white-noise-solu} and Lemma \ref{lem-identity},
  \begin{equation}\label{simplified-eq}
  \aligned
  \big\<\omega^{N_i}_t, \phi\big\> &= \big\<\omega^{N_i}_0, \phi\big\> + \int_0^t \big\<\omega^{N_i}_r\otimes  \omega^{N_i}_r, H_\phi\big\>\, \d r + \nu \int_0^t \big\< \omega^{N_i}_r, \Delta \phi \big\>\,\d r \\
  &\hskip13pt - 2 \sqrt{2\nu}\, \varepsilon_{N_i} \sum_{k\in \Lambda_{N_i}} \int_0^t \big\< \omega^{N_i}_r,\sigma_k \cdot \nabla \phi \big\>\,\d W^k_r.
  \endaligned
  \end{equation}
For any $\phi\in C^\infty(\T^2)$, let $\big\{H^n_\phi \big\}_{n\geq 1}\subset H^{2+}(\T^2\times \T^2)$ be an approximation of $H_\phi$ satisfying (cf. \cite[Remark 9]{F1})
  $$\lim_{n\to\infty} \int_{\T^2} \int_{\T^2} \big(H^n_\phi - H_\phi \big)^2(x,y) \,\d x\d y =0 \quad \mbox{and} \quad \int_{\T^2} H^n_\phi(x,x) \,\d x =0, \quad n\geq 1.$$
Note that $\big(\tilde \omega^{N_i}_\cdot, \tilde W^{N_i}_\cdot\big)$ has the same law as $\big(\omega^{N_i}_\cdot, W_\cdot \big)$, and the latter satisfies the equation \eqref{simplified-eq}, therefore, it is easy to show that
  \begin{equation*}
  \aligned &\, \tilde \E\bigg\{\sup_{t\in [0,T]} \bigg| \big\< \tilde \omega^{N_i}_t, \phi\big\> - \big\< \tilde \omega^{N_i}_0, \phi\big\> - \int_0^t \big\<\tilde \omega^{N_i}_r\otimes  \tilde \omega^{N_i}_r, H_\phi\big\>\, \d r - \nu \int_0^t \big\< \tilde \omega^{N_i}_r, \Delta \phi \big\>\,\d r\\
  &\hskip50pt + 2 \sqrt{2\nu}\, \varepsilon_{N_i} \sum_{k\in \Lambda_{N_i}} \int_0^t \big\< \tilde \omega^{N_i}_r,\sigma_k \cdot \nabla \phi \big\>\,\d \tilde W^{N_i,k}_r \bigg| \bigg\} \\
  \leq &\, \tilde \E\bigg\{\sup_{t\in [0,T]} \bigg| \int_0^t \big\<\tilde \omega^{N_i}_r\otimes  \tilde \omega^{N_i}_r, H_\phi - H^n_\phi \big\>\, \d r \bigg| \bigg\} + \E\bigg\{\sup_{t\in [0,T]} \bigg| \int_0^t \big\< \omega^{N_i}_r\otimes  \omega^{N_i}_r, H_\phi - H^n_\phi \big\>\, \d r \bigg| \bigg\},
  \endaligned
  \end{equation*}
which, since both $\tilde \omega^{N_i}_r$ and $\omega^{N_i}_r$ are distributed as the white noise measure $\mu$, is dominated by
  $$2 T \E_\mu \big| \big\<\omega \otimes \omega, H_\phi - H^n_\phi \big\> \big| \leq 2\sqrt{2}\, T \bigg(\int_{\T^2} \int_{\T^2} \big(H^n_\phi - H_\phi \big)^2(x,y) \,\d x\d y\bigg)^{1/2},$$
where the inequality can be found in the proof of \cite[Theorem 8]{F1}. Letting $n\to \infty$ yields, $\tilde\P$-a.s., for all $t\in [0,T]$,
  \begin{equation}\label{eq-seq}
  \aligned
  \big\<\tilde \omega^{N_i}_t, \phi\big\> &= \big\<\tilde \omega^{N_i}_0, \phi\big\> + \int_0^t \big\<\tilde \omega^{N_i}_r\otimes \tilde \omega^{N_i}_r, H_\phi\big\>\, \d r + \nu \int_0^t \big\< \tilde \omega^{N_i}_r, \Delta \phi \big\>\,\d r \\
  &\hskip13pt - 2 \sqrt{2\nu}\, \varepsilon_{N_i} \sum_{k\in \Lambda_{N_i}} \int_0^t \big\< \tilde \omega^{N_i}_r,\sigma_k \cdot \nabla \phi \big\>\,\d \tilde W^{{N_i}, k}_r.
  \endaligned
  \end{equation}

\begin{remark}
Using the a.s. convergence of $\tilde \omega^{N_i}$ to $\tilde \omega$ in $C\big([0,T], H^{-1-}(\T^2)\big)$, we can show that the quantities in the first line of \eqref{eq-seq} converge respectively in $L^2\big(\tilde \Theta, \tilde \P \big)$ to
  $$\<\tilde \omega_t, \phi\>,\quad \<\tilde \omega_0, \phi\> ,\quad \int_0^t \big\<\tilde \omega_r\otimes \tilde \omega_r, H_\phi\big\>\, \d r,\quad \int_0^t \< \tilde \omega_r, \Delta \phi \>\,\d r,$$
see \cite[Proposition 3.6]{FL-1} for details. However, the term involving stochastic integrals does not converge strongly to the last term of \eqref{vorticity-NSE}. Therefore, we can only seek for a weaker form of convergence.
\end{remark}

Before proceeding further, we introduce some notations. By $\Lambda\Subset \Z_0^2$ we mean that $\Lambda$ is a finite set. Let $\Pi_\Lambda: H^{-1-}(\T^2) \to \text{span}\{e_k: k\in \Lambda\}$ be the projection operator: $\Pi_\Lambda\omega = \sum_{l\in \Lambda} \<\omega, e_l\> e_l$. We shall use the family of cylindrical functions below:
  $$\mathcal{FC}_b^2 = \big\{F(\omega)= f(\<\omega, e_l\>; l\in \Lambda) \mbox{ for some } \Lambda\Subset \Z_0^2 \mbox{ and } f\in C_b^2\big(\R^\Lambda\big)\big\},$$
where $\R^\Lambda$ is the $(\#\Lambda)$-dimensional Euclidean space. To simplify the notations, sometimes we write the cylindrical functions as $F= f\circ \Pi_\Lambda$, and for $l,m\in \Lambda$, $f_l(\omega) = (\partial_l f)(\Pi_\Lambda \omega)$ and $f_{l,m}(\omega) = (\partial_l\partial_m f)(\Pi_\Lambda \omega)$. Denote by $\L_\infty$ the generator of the equation \eqref{vorticity-NSE}: for any cylindrical function $F= f \circ \Pi_\Lambda$ with $\Lambda \Subset \Z_0^2$,
  \begin{equation}\label{generator}
  \L_\infty F= 4 \nu \pi^2 \sum_{l\in \Lambda} |l|^2 \big[f_{l,l}(\omega) - f_l(\omega) \< \omega, e_l \>\big] - \< u(\omega)\cdot \nabla\omega, D F\>,
  \end{equation}
where the drift part
  $$\< u(\omega)\cdot \nabla\omega, D F\>= - \sum_{l\in \Lambda} f_l(\omega) \big\<\omega\otimes \omega, H_{e_l}\big\>.$$
Finally we introduce the notation
  \begin{equation}\label{coefficients}
  C_{k,l} = \frac{k^\perp \cdot l}{|k|^2} ,\quad k,l \in\Z_0^2.
  \end{equation}

Now we prove that the limit $\tilde\omega$ is a martingale solution of the operator $\L_\infty$.

\begin{proposition}\label{approx-prop-1}
For any $F\in \mathcal{FC}_b^2$,
  \begin{equation}\label{approx-prop-1.1}
  \tilde M^F_t:= F(\tilde \omega_t ) -F(\tilde \omega_0 ) - \int_0^t \L_\infty F(\tilde \omega_s ) \,\d s
  \end{equation}
is an $\tilde{\mathcal F}_t = \sigma(\tilde \omega_s: s\leq t)$-martingale.
\end{proposition}

\begin{proof}
Recall the CONS defined in \eqref{ONB}. Taking $\phi= e_l$ in \eqref{eq-seq} for some $l\in \Z_0^2$, we have
  \begin{equation}\label{eq-seq-1}
  \aligned
  \d \big\< \tilde \omega^{N_i}_t, e_l\big\>&= \big\<\tilde \omega^{N_i}_t\otimes \tilde \omega^{N_i}_t, H_{e_l}\big\>\, \d t -4\nu \pi^2 |l|^2 \big\< \tilde \omega^{N_i}_t, e_l \big\>\,\d t \\
  &\hskip13pt - 2 \sqrt{2\nu}\, \varepsilon_{N_i} \sum_{k\in \Lambda_{N_i}} \big\< \tilde \omega^{N_i}_t,\sigma_k \cdot \nabla e_l \big\>\,\d \tilde W^{{N_i}, k}_t.
  \endaligned
  \end{equation}
Therefore, for $l,m \in \Z_0^2$,
  \begin{equation*}
  \d \big\< \tilde \omega^{N_i}_t, e_l\big\> \cdot \d \big\< \tilde \omega^{N_i}_t, e_m\big\>= 8\nu \varepsilon_{N_i}^2 \sum_{k\in \Lambda_{N_i}} \big\< \tilde \omega^{N_i}_t,\sigma_k \cdot \nabla e_l \big\> \big\< \tilde \omega^{N_i}_t,\sigma_k \cdot \nabla e_m \big\>\,\d t.
  \end{equation*}
Direct computation leads to $\sigma_k \cdot \nabla e_l = \sqrt{2} \pi C_{k,l} e_k e_{-l}$; hence
  $$\aligned
  \big\< \tilde \omega^{N_i}_t,\sigma_k \cdot \nabla e_l \big\> \big\< \tilde \omega^{N_i}_t,\sigma_k \cdot \nabla e_m \big\> &= 2\pi^2 C_{k,l}C_{k,m} \big\< \tilde \omega^{N_i}_t, e_k e_{-l}\big\> \big\< \tilde \omega^{N_i}_t,e_k e_{-m} \big\> \\
  &= 2\pi^2 C_{k,l}C_{k,m} \Big[\big\< \tilde \omega^{N_i}_t, e_k e_{-l}\big\> \big\< \tilde \omega^{N_i}_t,e_k e_{-m} \big\> -\delta_{l,m} \Big]\\
  &\hskip13pt + 2\pi^2 \delta_{l,m} C_{k,l}^2.
  \endaligned$$
As a result,
  \begin{equation*}
  \aligned
  \d \big\< \tilde \omega^{N_i}_t, e_l\big\> \cdot \d \big\< \tilde \omega^{N_i}_t, e_m\big\> &= 16\nu \pi^2 \varepsilon_{N_i}^2 \sum_{k\in \Lambda_{N_i}} C_{k,l}C_{k,m} \Big[\big\< \tilde \omega^{N_i}_t, e_k e_{-l}\big\> \big\< \tilde \omega^{N_i}_t,e_k e_{-m} \big\> -\delta_{l,m} \Big] \d t \\
  &\hskip13pt + 8\nu \pi^2 \delta_{l,m} |l|^2 \,\d t,
  \endaligned
  \end{equation*}
where in the last step we have used Lemma \ref{lem-2-1}. To simplify the notations, we denote by
  $$R_{l,m}\big(\tilde \omega^{N_i}_t\big) = 8\nu \pi^2 \sum_{k\in \Lambda_{N_i}} C_{k,l}C_{k,m} \Big[\big\< \tilde \omega^{N_i}_t, e_k e_{-l}\big\> \big\< \tilde \omega^{N_i}_t,e_k e_{-m} \big\> -\delta_{l,m} \Big].$$
Recall that $\tilde \omega^{N_i}_t$ is a white noise for any $t\in [0,T]$, thus by the second assertion of Proposition \ref{prop-operator}, $R_{l,m}\big(\tilde \omega^{N_i}_t\big)$ is bounded in any $L^p\big([0,T]\times \tilde\Theta\big),\, p>1$. Finally, we get
  \begin{equation}\label{eq-seq-2}
  \aligned
  \d \big\< \tilde \omega^{N_i}_t, e_l\big\> \cdot \d \big\< \tilde \omega^{N_i}_t, e_m\big\> &= 2 \varepsilon_{N_i}^2 R_{l,m}\big(\tilde \omega^{N_i}_t\big) \,\d t + 8\nu \pi^2 \delta_{l,m} |l|^2 \,\d t.
  \endaligned
  \end{equation}

By the It\^o formula and \eqref{eq-seq-1}, \eqref{eq-seq-2},
  \begin{equation*}
  \aligned
  \d F\big(\tilde \omega^{N_i}_t \big)= &\ \d f\big(\big\< \tilde \omega^{N_i}_t, e_l\big\>; l\in \Lambda\big)\\
  =&\ \sum_{l\in \Lambda} f_l\big(\tilde \omega^{N_i}_t \big) \Big[\big\<\tilde \omega^{N_i}_t\otimes \tilde \omega^{N_i}_t, H_{e_l}\big\> -4\nu \pi^2 |l|^2 \big\< \tilde \omega^{N_i}_t, e_l \big\>\Big]\,\d t \\
  &\ - 2 \sqrt{2\nu}\, \varepsilon_{N_i} \sum_{l\in \Lambda} f_l\big(\tilde \omega^{N_i}_t \big) \sum_{k\in \Lambda_{N_i}} \big\< \tilde \omega^{N_i}_t,\sigma_k \cdot \nabla e_l \big\>\,\d \tilde W^{{N_i}, k}_t \\
  &\ + \sum_{l, m\in \Lambda} f_{l,m}\big(\tilde \omega^{N_i}_t \big) \big[\varepsilon_{N_i}^2 R_{l,m}\big(\tilde \omega^{N_i}_t\big) + 4\nu \pi^2 \delta_{l,m} |l|^2 \big] \,\d t.
  \endaligned
  \end{equation*}
Recalling the operator $\L_\infty$ defined in \eqref{generator}, the above formula can be rewritten as
  \begin{equation}\label{Ito-formula}
  \d F\big(\tilde \omega^{N_i}_t \big) = \L_\infty F\big(\tilde \omega^{N_i}_t \big) \,\d t + \varepsilon_{N_i}^2 \tilde \zeta^{N_i}_t \,\d t + \d \tilde M^{N_i}_t,
  \end{equation}
where
  $$\tilde \zeta^{N_i}_t = \sum_{l, m\in \Lambda} f_{l,m}\big(\tilde \omega^{N_i}_t \big) R_{l,m}\big(\tilde \omega^{N_i}_t\big)$$
is bounded in $L^p\big([0,T]\times \tilde\Theta \big)$ for any $p>1$, and the martingale part
  $$\d \tilde M^{N_i}_t= -2 \sqrt{2\nu}\, \varepsilon_{N_i} \sum_{l\in \Lambda} f_l\big(\tilde \omega^{N_i}_t \big) \sum_{k\in \Lambda_{N_i}} \big\< \tilde \omega^{N_i}_t,\sigma_k \cdot \nabla e_l \big\>\,\d \tilde W^{{N_i}, k}_t. $$
Note that $\tilde M^{N_i}_t$ is a martingale w.r.t. the filtration
  $$\tilde{\mathcal F}^{N_i}_t = \sigma\big(\tilde \omega^{N_i}_s, \tilde W^{N_i}_s: s\leq t\big),$$
where we denote by $\tilde W^{N_i}_s = \big\{ \tilde W^{N_i,k}_s \big\}_{k\in \Z_0^2}$.

Next, we show that the formula \eqref{Ito-formula} converges as $i\to \infty$ in a suitable sense. To this end, we follow the argument of \cite[p. 232]{DaPZ}. Fix any $0<s <t\leq T$. Take a real valued, bounded and continuous function $\varphi: C\big([0,s], H^{-1-} \times \R^{\Z_0^2} \big)\to \R$. By \eqref{Ito-formula}, we have
  $$\tilde \E\bigg[\bigg( F\big(\tilde \omega^{N_i}_t \big) -F\big(\tilde \omega^{N_i}_s \big) - \int_s^t \L_\infty F\big(\tilde \omega^{N_i}_r \big) \,\d r - \varepsilon_{N_i}^2 \int_s^t \tilde \zeta^{N_i}_r \,\d r \bigg) \varphi\big(\tilde \omega^{N_i}_\cdot, \tilde W^{N_i}_\cdot \big)\bigg] =0.$$
Since $F\in \mathcal{FC}_b^2$ and $\tilde \omega^{N_i}_t$ is a white noise, all the terms in the bracket belong to $L^p\big(\tilde \P\big)$ for any $p>1$. Recalling that, $\tilde\P$-a.s., $\big(\tilde \omega^{N_i}_\cdot, \tilde W^{{N_i}}_\cdot \big)$ converges to $\big(\tilde \omega_\cdot, \tilde W_\cdot \big)$ in $C\big([0,T], H^{-1-} \times \R^{\Z_0^2} \big)$, thus, letting $i\to \infty$ in the above equality yields
  $$\tilde \E\bigg[\bigg( F(\tilde \omega_t ) -F(\tilde \omega_s ) - \int_s^t \L_\infty F(\tilde \omega_r ) \,\d r \bigg) \varphi\big(\tilde \omega_\cdot, \tilde W_\cdot \big)\bigg] =0.$$
The arbitrariness of $0<s<t$ and $\varphi: C\big([0,s], H^{-1-} \times \R^{\Z_0^2} \big)\to \R$ implies that $\tilde M^F_\cdot$ is a martingale with respect to the filtration $\tilde{\mathcal G}_t = \sigma\big(\tilde \omega_s, \tilde W_s: s\leq t\big),\, t\in [0,T]$. For any $0\leq s< t\leq T$, we have $\tilde{\mathcal F}_s\subset \tilde{\mathcal G}_s$, thus
  $$\tilde \E \big(\tilde M^F_t \big|\tilde{\mathcal F}_s \big)= \tilde \E \Big[\tilde \E \big(\tilde M^F_t \big| \tilde{\mathcal G}_s \big) \big| \tilde{\mathcal F}_s \Big] = \tilde \E \big[\tilde M^F_s \big| \tilde{\mathcal F}_s \big] =\tilde M^F_s ,$$
since $\tilde M^F_s$ is adapted to $\tilde{\mathcal F}_s $.
\end{proof}

Next we show that $\tilde\omega_\cdot$ solves \eqref{vorticity-NSE} in a weak sense, cf. \cite[Definition 4.1]{DaPD}.

\begin{proposition}
There exists a family of independent standard Brownian motions $\big\{ \tilde W^k_t: t\geq 0\big\}_{k\in \Z_0^2}$ such that $(\tilde \omega_\cdot, \tilde W_\cdot)$ solves \eqref{vorticity-NSE}, where $\tilde W_t =\sum_{k\in \Z_0^2} \tilde W^{-k}_t e_k \frac{k^\perp}{|k|}$.
\end{proposition}

\begin{proof}
In order to identify the process $\tilde \omega_t$, we take some special cylinder functions $F$. First, let $F(\omega)= \<\omega, e_l\>$ for some $l\in \Z_0^2$, then
  $$\L_\infty F(\omega)= -4\nu \pi^2 |l|^2\<\omega, e_l\> - \<u(\omega)\cdot \nabla\omega, e_l\>.$$
Thus, by Proposition \ref{approx-prop-1}, we have the martingales
  $$\tilde M^{(l)}_t:= \<\tilde \omega_t, e_l\> -\<\tilde \omega_0, e_l\> + \int_0^t \big(4\nu \pi^2 |l|^2\<\tilde \omega_s, e_l\> + \<u(\tilde \omega_s)\cdot \nabla\tilde \omega_s, e_l\>\big) \,\d s,\quad l\in \Z^2_0.$$
In particular,
  \begin{equation}\label{Ito-differential}
  \d \<\tilde \omega_t, e_l\> = \d \tilde M^{(l)}_t - \big(4\nu \pi^2 |l|^2\<\tilde \omega_t, e_l\> + \<u(\tilde \omega_t)\cdot \nabla\tilde \omega_t, e_l\>\big) \,\d t,\quad l\in \Z^2_0.
  \end{equation}
Therefore, for $l,m\in \Z_0^2,\, l\neq m$,
  $$\aligned
  \d [\<\tilde \omega_t, e_l\> \<\tilde \omega_t, e_m\>] &= \<\tilde \omega_t, e_m\> \d \tilde M^{(l)}_t - \<\tilde \omega_t, e_m\> \big(4\nu \pi^2 |l|^2\<\tilde \omega_t, e_l\> + \<u(\tilde \omega_t)\cdot \nabla\tilde \omega_t, e_l\>\big) \,\d t \\
  &\hskip13pt + \<\tilde \omega_t, e_l\> \d \tilde M^{(m)}_t - \<\tilde \omega_t, e_l\> \big(4\nu \pi^2 |m|^2\<\tilde \omega_t, e_m\> + \<u(\tilde \omega_t)\cdot \nabla\tilde \omega_t, e_m\>\big) \,\d t \\
  &\hskip13pt + \d \big\<\tilde M^{(l)}, \tilde M^{(m)}\big\>_t.
  \endaligned$$
Equivalently, denoting by $\tilde M_t$ the martingale part,
  \begin{equation}\label{martingale-1}
  \aligned
  \<\tilde \omega_t, e_l\> \<\tilde \omega_t, e_m\> &= \<\tilde \omega_0, e_l\> \<\tilde \omega_0, e_m\> + \tilde M_t -4\nu \pi^2 (|l|^2+|m|^2) \int_0^t \<\tilde \omega_s, e_l\> \<\tilde \omega_s, e_m\> \,\d s \\
  &\hskip13pt - \int_0^t \big[\<\tilde \omega_s, e_m\> \<u(\tilde \omega_s)\cdot \nabla \tilde \omega_s, e_l\> + \<\tilde \omega_s, e_l\> \<u(\tilde \omega_s)\cdot \nabla \tilde \omega_s, e_m\> \big]\,\d s\\
  & \hskip13pt + \big\<\tilde M^{(l)}, \tilde M^{(m)}\big\>_t.
  \endaligned
  \end{equation}
On the other hand, taking $F(\omega)= \<\omega, e_l\> \<\omega, e_m\>$, we have
  $$\aligned
  \L_\infty F(\omega)&= \<\omega, e_m\> \big(-4\nu \pi^2 |l|^2\<\omega, e_l\> - \<u(\omega)\cdot \nabla\omega, e_l\>\big)\\
  &\hskip13pt + \<\omega, e_l\> \big(-4\nu \pi^2 |m|^2\<\omega, e_m\> - \<u(\omega)\cdot \nabla\omega, e_m\>\big)\\
  &= -4\nu \pi^2 (|l|^2+|m|^2)\<\omega, e_l\> \<\omega, e_m\> - \<\omega, e_m\> \<u(\omega)\cdot \nabla\omega, e_l\> - \<\omega, e_l\> \<u(\omega)\cdot \nabla\omega, e_m\>.
  \endaligned$$
Therefore, by \eqref{approx-prop-1.1}, we also have the martingale
  $$\aligned
  \tilde M^{(l,m)}_t&= \<\tilde \omega_t, e_l\> \<\tilde \omega_t, e_m\> - \<\tilde \omega_0, e_l\> \<\tilde \omega_0, e_m\> + 4\nu \pi^2 (|l|^2+|m|^2) \int_0^t \<\tilde \omega_s, e_l\> \<\tilde \omega_s, e_m\> \,\d s \\
  &\hskip13pt + \int_0^t \big[\<\tilde \omega_s, e_m\> \<u(\tilde \omega_s)\cdot \nabla \tilde \omega_s, e_l\> + \<\tilde \omega_s, e_l\> \<u(\tilde \omega_s)\cdot \nabla \tilde \omega_s, e_m\> \big]\,\d s.
  \endaligned$$
Comparing this equality with \eqref{martingale-1}, we obtain
  \begin{equation}\label{quadratic-var}
  \big\<\tilde M^{(l)}, \tilde M^{(m)}\big\>_t = 0, \quad l\neq m.
  \end{equation}

Next, by \eqref{Ito-differential}, we have
  $$\d\big(\<\tilde \omega_t, e_l\>^2 \big) = 2 \<\tilde \omega_t, e_l\> \big[\d \tilde M^{(l)}_t - \big(4\nu \pi^2 |l|^2\<\tilde \omega_t, e_l\> + \<u(\tilde \omega_t)\cdot \nabla\tilde \omega_t, e_l\>\big) \,\d t \big]+ \d \big\< \tilde M^{(l)}\big\>_t,$$
which implies
  \begin{equation}\label{martingale-2}
  \aligned
  \<\tilde \omega_t, e_l\>^2& = \<\tilde \omega_0, e_l\>^2 +2 \int_0^t \<\tilde \omega_s, e_l\> \,\d \tilde M^{(l)}_s + \big\< \tilde M^{(l)}\big\>_t \\
  &\hskip13pt - 2 \int_0^t \<\tilde \omega_s, e_l\> \big(4\nu \pi^2 |l|^2\<\tilde \omega_s, e_l\> + \<u(\tilde \omega_s)\cdot \nabla\tilde \omega_s, e_l\>\big) \,\d s.
  \endaligned
  \end{equation}
Similarly, taking $F(\omega) = \<\omega, e_l\>^2$, one has
  $$\L_\infty F(\omega)= -2 \<\omega, e_l\> \<u(\omega)\cdot \nabla\omega, e_l\> - 8\nu \pi^2 |l|^2 \big(\< \omega, e_l\>^2 -1 \big).$$
Substituting this into \eqref{approx-prop-1.1} gives us the martingale
  $$\tilde M^{(l,l)}_t= \<\tilde \omega_t, e_l\>^2 - \<\tilde \omega_0, e_l\>^2 + 2 \int_0^t \<\tilde \omega_s, e_l\>\<u(\tilde \omega_s)\cdot \nabla\tilde \omega_s, e_l\> \,\d s + 8\nu \pi^2 |l|^2 \int_0^t \big(\<\tilde \omega_s, e_l\>^2 -1 \big)\,\d s. $$
Comparing this identity with \eqref{martingale-2} yields
  \begin{equation}\label{quadratic-var-1}
  \big\< \tilde M^{(l)}\big\>_t = 8\nu \pi^2 |l|^2 t.
  \end{equation}

According to the equalities \eqref{quadratic-var} and \eqref{quadratic-var-1}, if we define
  $$\tilde W^l_t= \frac1{2\sqrt{2\nu}\pi |l|} \tilde M^{(l)}_t, \quad l\in \Z_0^2.$$
Then $\big\{\tilde W^l\big\}_{l\in \Z_0^2}$ is a family of independent standard Brownian motions. Now the formula \eqref{Ito-differential} becomes
  \begin{equation*}
  \d \<\tilde \omega_t, e_l\> = 2\sqrt{2\nu}\pi |l|\,\d \tilde W^l_t - \big(4\nu \pi^2 |l|^2\<\tilde \omega_t, e_l\> + \<u(\tilde \omega_t)\cdot \nabla\tilde \omega_t, e_l\>\big) \,\d t,\quad l\in \Z^2_0.
  \end{equation*}
The above equations are the component form of the equation below
  \begin{equation}\label{Ito-differential-1}
  \d \tilde \omega_t + u(\tilde \omega_t)\cdot \nabla\tilde \omega_t\,\d t = \nu \Delta \tilde \omega_t\,\d t + \sqrt{2\nu}\, \nabla^\perp \cdot \d \tilde W_t,
  \end{equation}
where $\tilde W_t$ is the vector valued white noise  defined in the statement of the proposition. Therefore, $\tilde \omega_t$ solves the vorticity form of the Navier--Stokes equation driven by space-time white noise.
\end{proof}

We can rewrite \eqref{Ito-differential-1} in the velocity-pressure variables as follows:
  \begin{equation}\label{NSE}
  \d\tilde u + \tilde u\cdot\nabla \tilde u \,\d t + \nabla \tilde p\,\d t = \nu \Delta \tilde u\,\d t + \sqrt{2\nu}\, \d \tilde W.
  \end{equation}
An $L^1$-uniqueness result was proved in \cite{Sauer} for the Kolmogorov operator $\mathcal L_{NS}$ associated to \eqref{NSE}, but on the torus $[0,2\pi]^2= \R^2/(2\pi \Z^2)$. In order to apply this result, we need to transform our equation to that case. Let $\mathcal H$ be the  subspace of $L^2\big([0,2\pi]^2, \R^2 \big)$ consisting of periodic and divergence free vector fields with vanishing mean.

\begin{lemma}\label{lem-NSE}
For $(t,\xi)\in \R_+\times [0,2\pi]^2$, let $u(t,\xi)= 2\pi\, \tilde u(t, \xi/(2\pi))$, $p(t,\xi)= 4\pi^2 \tilde p(t, \xi/(2\pi))$ and $W(t,\xi)= (2\pi)^{-1} \tilde W(t, \xi/(2\pi))$. Then $W(t,\xi)$ is a cylindrical Brownian motion on $\mathcal H$ and
  $$\d u + u\cdot \nabla u\,\d t+ \nabla p\,\d t = 4\pi^2 \nu \Delta u\,\d t + 4\sqrt{2\pi^4\nu}\, \d W.$$
\end{lemma}

\begin{proof}
For $l\in \Z_0^2$, set
  $$v_l(\xi)= \frac1{\sqrt{2}\,\pi} \frac{l^\perp}{|l|} \begin{cases}
  \cos(l\cdot \xi), & l\in \Z^2_+;\\
  \sin(l\cdot \xi), & l\in \Z^2_-.
  \end{cases}$$
Then $\{v_l\}_{l\in \Z_0^2}$ is a CONS of $\mathcal H$ and
  \begin{equation}\label{lem-NSE.1}
  W(t,\xi) = \frac1{2\pi}\tilde W(t, \xi/(2\pi)) = \sum_{l\in\Z_0^2} \tilde W^{-l}(t) v_l(\xi).
  \end{equation}
Since $\big\{\tilde W^{-l}\big\}_{l\in \Z_0^2}$ is a family of independent standard Brownian motions, we obtain the first result. The second assertion follows from \eqref{NSE} and the definitions of $u,p$ and $W$:
  \[\aligned
  \d u + u\cdot \nabla u\,\d t+ \nabla p \,\d t&= 2\pi \big[\d \tilde u + \tilde u\cdot \nabla \tilde u \,\d t+ \nabla \tilde p\,\d t\big] (t, \xi/(2\pi)) \\
  &= 2\pi \big[\nu \Delta \tilde u + \sqrt{2\nu}\, \d\tilde{W} \big] (t, \xi/(2\pi)) \\
  &= 4\pi^2 [\nu \Delta u + \sqrt{2\nu}\, \d  W].
  \endaligned \]
The proof is complete.
\end{proof}

Recall that $\omega^N_t$ is the white noise solution of \eqref{vorticity-Euler-1}, and $\{Q^N\}_{N\geq 1}$ are the distributions of $\big(\omega^N_t \big)_{0\leq t\leq T}$ on $C\big([0,T], H^{-1-}(\T^2) \big)$. Now we can prove the main result of this paper.

\begin{theorem}\label{thm-convergence}
Denote by
  $$S= \sum_{k\in\Z_0^2} \frac1{|k|^4} <+\infty$$
and assume that
  $$\nu > \frac{2\sqrt{5S}}{\pi^2}.$$
Then the whole sequence $\{Q^N\}_{N\geq 1}$ converges weakly to the distribution of solution to \eqref{Ito-differential-1}.
\end{theorem}

\begin{proof}
Substitute $\nu$ and $\sigma$ in \cite[(2)]{Sauer} by $4\pi^2\nu$ and $4\sqrt{2\pi^4\nu}$, respectively, and take $C=0$ (i.e. the Coriolis force vanishes). Note that the measure $\mu_{\sigma, \nu}$ defined in \cite[(4)]{Sauer} coincides with $\mathcal N\big(0, 4\pi^2 A^{-1} \big)$, where $A$ is the Stokes operator. Under our condition, Assumption A on p. 572 of \cite{Sauer} is satisfied, thus by Corollary 1 on the same page, the operator $(\L_{NS}, \mathcal{FC}_b^2)$ is $L^1$-unique. Here, by an abuse of notation, we denote also by $\mathcal{FC}_b^2$ the cylindrical functions corresponding to the Navier--Stokes equation driven by space-time white noise. This implies that its closure $\big(\overline{\L_{NS}}, D(\overline{\L_{NS}} \,) \big)$ generates a $C_0$-semigroup of contractions $\{P_t\}_{t\geq 0}$ in $L^1(\mu_{\sigma, \nu})$ and $\mu_{\sigma, \nu}$ is invariant for $P_t$. According to \cite[Remark 1.2]{Stannat}, the martingale problem associated to $(\L_{NS}, \mathcal{FC}_b^2)$ has a unique solution. This implies the uniqueness of the martingale solution to the original operator $(\L_\infty, \mathcal{FC}_b^2)$ associated to \eqref{Ito-differential-1}.

Recall that we have shown the tightness of the family $\{Q^N\}_{N\geq 1}$. Thus we deduce the assertion from the uniqueness of martingale problem associated to $(\L_\infty, \mathcal{FC}_b^2)$.
\end{proof}

\begin{remark}
\begin{itemize}
\item[\rm (i)] We have
  $$S= \sum_{k\in\Z_0^2} \frac1{|k|^4} \approx \int_{|x|\geq 1/2} \frac{\d x}{|x|^4} = \int_{1/2}^\infty 2\pi \frac{\d r}{r^3} = 4\pi.$$
Hence,
  $$\frac{2\sqrt{5S}}{\pi^2} \approx \frac{4\sqrt{5}}{\pi^{3/2}} \approx 1.6062760546.$$

\item[\rm (ii)] For other weaker uniqueness results on $(\L_{NS}, \mathcal{FC}_b^2)$, see e.g. \cite{AlFe, Stannat03, AlBaFe}. A similar $L^1$-uniqueness result was proved in \cite[Theroem 1.1]{Stannat} with less precise estimate on the lower bound of $\nu$.
\end{itemize}
\end{remark}

\section{The Kolmogorov equation corresponding to \eqref{vorticity-NSE}}\label{sec-Kolmogorov-eq}

The purpose of this section is to solve the Kolmogorov equation associated to the vorticity form of the Navier--Stokes equation \eqref{vorticity-NSE} driven by the space-time white noise. To simplify notations we write $H^{-1-}$ instead of $H^{-1-}(\T^2)$. The main result is

\begin{theorem}\label{main-result}
Let $\rho_0 \in L^2(H^{-1-},\mu)$. Then there exists $\rho \in L^\infty \big(0,T; L^2(H^{-1-},\mu)\big)$ which solves
  \begin{equation}\label{modified-eq}
  \partial_t \rho_t = \L_\infty^\ast \rho_t, \quad \rho|_{t=0} = \rho_0.
  \end{equation}
More precisely, for any cylindrical function $F= f\circ \Pi_\Lambda$ and $\alpha \in C^1 \big([0,T],\R \big)$ satisfying $\alpha(T)=0$, one has
  \begin{equation}\label{main-result.2}
  \aligned
  0= &\ \alpha(0)\int F \rho_0 \,\d\mu + \int_0^T \!\! \int \rho_t \big( \alpha'(t) F - \alpha(t) \< u(\omega) \cdot\nabla\omega , D F\> \big) \,\d\mu\d t \\
  & + 4 \nu \pi^2 \sum_{l\in \Lambda} |l|^2 \int_0^T \!\! \int \alpha(t)\, \rho_t\, \big[f_{l,l}(\omega) - f_l(\omega) \< \omega, e_l \>\big] \,\d\mu\d t.
  \endaligned
  \end{equation}
\end{theorem}

\begin{remark}
Unlike \cite[Theorem 1.1]{FL}, we do not have result on $\<\sigma_k \cdot \nabla\omega , D \rho_t \>$, see Remark \ref{rem-1} below for details.
\end{remark}

We can prove Theorem \ref{main-result} by following the line of arguments in \cite{FL}. Due to a technical problem which will become clear in the proof of Theorem \ref{main-result}, as in \cite[Section 4]{FL}, we consider an equation slightly different from \eqref{vorticity-Euler-1}:
  \begin{equation}\label{new-approx-eq}
  \d \omega^N_t + u^N_t\cdot \nabla\omega^N_t = 2\sqrt{2\nu}\, \tilde \eps_N \sum_{k\in \Gamma_N} \sigma_k\cdot \nabla\omega^N_t \circ\d W^k_t.
  \end{equation}
where $\Gamma_N = \{k\in \Z_0^2: |k|\leq N/3\}$ and
  \begin{equation}\label{new-constant}
  \tilde \eps_N= \bigg(\sum_{k\in \Gamma_N} \frac1{|k|^2}\bigg)^{-1/2}.
  \end{equation}
The generator of \eqref{new-approx-eq} is
  \begin{equation}\label{approx-generator}
  \L_N F(\omega)= 4\nu \tilde \eps_N^2 \sum_{k\in \Gamma_N} \big\<\sigma_k\cdot \nabla\omega, D \<\sigma_k\cdot \nabla\omega, D F\>\big\> - \< u(\omega)\cdot \nabla\omega, D F\>, \quad F\in \mathcal{FC}_b^2.
  \end{equation}

Now we need the decomposition formula proved in Proposition \ref{prop-operator} (replacing $\Lambda_N$ there by $\Gamma_N$). For any cylindrical function $F \in \mathcal{FC}_b^2$, we denote by
  \begin{equation}\label{diffusion-part}
  \L_N^0 F(\omega) := \frac12 \sum_{k\in \Gamma_N} \big\< \sigma_k\cdot \nabla\omega, D \< \sigma_k\cdot \nabla\omega, D F\> \big\>,
  \end{equation}
then
  $$\L_N F(\omega)= 8\nu \tilde \eps_N^2 \L_N^0 F(\omega) - \< u(\omega)\cdot \nabla\omega, D F\>.$$
The assertions of Proposition \ref{prop-operator} immediately yield

\begin{proposition}\label{prop-1}
For any $F \in \mathcal{FC}_b^2$, it holds that
  $$ \lim_{N\to \infty} \L_N F = \L_\infty F \quad \mbox{in } L^2\big(H^{-1-}, \mu\big).$$
\end{proposition}

With this result in hand, we will define the Galerkin approximation of the operator $\L_\infty$ for which we need some notations (see \cite[Section 2]{FL} for details). Let $H_N= \mbox{span}\{e_k:k\in\Lambda_N\}$ and $\Pi_N: H^{-1-}(\T^2) \to H_N$ be the projection operator, which is an orthogonal projection when restricted to $L^2(\T^2)$. We project the drift term $u(\omega)\cdot \nabla\omega$ in \eqref{approx-generator} as follows:
  $$b_N(\omega):=\Pi_N \big(u(\Pi_N \omega) \cdot \nabla (\Pi_N \omega) \big),\quad \omega\in H^{-1-}(\T^2),$$
where $u(\Pi_N \omega)$ is obtained from the Biot--Savart law:
  $$u(\Pi_N \omega)(x) = \int_{\T^2} K(x-y)(\Pi_N \omega)(y)\,\d y.$$
We shall consider $b_N$ as a vector field on $H_N$ whose generic element is denoted by $\xi= \sum_{k\in \Lambda_N} \xi_k e_k$. Thus
  $$b_N(\xi)= \Pi_N \big(u(\xi) \cdot \nabla \xi\big), \quad \xi\in H_N.$$
Analogously, we define the projection of the diffusion coefficient $\sigma_k\cdot \nabla \omega$ in \eqref{approx-generator}:
  $$G_N^k(\xi)= \Pi_N\big(\sigma_k\cdot \nabla \xi\big), \quad \xi\in H_N.$$
It can be shown that $b_N$ and $G_N^k$ are divergence free with respect to the standard Gaussian measure $\mu_N$ on $H_N$. With the above preparations, we can define the Galerkin approximation of the operator $\L_\infty$ as
  \begin{equation*}
  \tilde \L_N \phi(\xi)= 4\nu \tilde \eps_N^2 \sum_{k\in \Gamma_N} \Big\<G_N^k, \nabla_N \big\<G_N^k, \nabla_N \phi\big\>_{H_N} \Big\>_{\! H_N}(\xi) - \<b_N, \nabla_N \phi\>_{H_N}(\xi).
  \end{equation*}

Consider the Kolmogorov equation on $H_N$:
  \begin{equation}\label{Kolmogorov-eq}
  \partial_t \rho^N_t = \tilde\L_N^\ast \rho^N_t, \quad \rho^N|_{t=0}= \rho^N_0\in C_b^2(H_N),
  \end{equation}
where $ \tilde\L_N^\ast$ is the adjoint operator of $ \tilde\L_N$ with respect to $\mu_N$. We slightly abuse the notation and denote by $\rho^N_t(\omega) = \rho^N_t(\Pi_N \omega)$, $N\geq 1$. It is easy to show that, for all $t\in [0,T]$,
  \begin{equation}\label{estimate-3}
  \big\|\rho^N_t \big\|_{L^2(\mu)}^2 + 8\nu \tilde\eps_N^2 \sum_{k\in \Gamma_N} \int_0^t \!\! \int_{H^{-1-}} \big\<\sigma_k \cdot \nabla(\Pi_N \omega), D \rho^N_s \big\>_{L^2(\T^2)}^2 \,\d\mu \d s = \big\|\rho^N_0 \big\|_{L^2(\mu)}^2.
  \end{equation}
For $k\notin \Gamma_N$, we set $\big\<\sigma_k \cdot \nabla(\Pi_N \omega), D \rho^N_s \big\>_{L^2(\T^2)}^2 \equiv 0$. Here are two simple observations.

\begin{proposition}\label{prop-convergence}
\begin{itemize}
\item[\rm (1)] $\big\{\rho^N \big\}_{N\in \N}$ is a bounded sequence in $L^\infty \big(0,T; L^2(H^{-1-}, \mu) \big)$;
\item[\rm (2)] the family
  $$\big\{ \tilde\eps_N \big\<\sigma_k \cdot \nabla(\Pi_N \omega), D \rho^N_t \big\>_{L^2(\T^2)}: (k, t, \omega) \in \Z^2_0 \times [0,T] \times H^{-1-} \big\}_{N \in \N}$$
is bounded in the Hilbert space $L^2\big(\Z^2_0 \times [0,T] \times H^{-1-}, \# \otimes \d t \otimes \mu\big)$, where $\#$ is the counting measure on $\Z^2_0$.
\end{itemize}
\end{proposition}

As a consequence, we obtain

\begin{corollary}\label{cor-1}
Assume $\rho_0\in L^2 ( H^{-1-}, \mu )$. Then the family $\big\{\rho^N \big\}_{N\in \N}$ has a subsequence which converges weakly-$\ast$ to some measurable function $\rho \in L^\infty \big(0,T; L^2( H^{-1-}, \mu) \big)$.
\end{corollary}

\begin{remark}\label{rem-1} {\rm
Unlike \cite[Theorem 3.2]{FL}, we are unable to show that $\<\sigma_k \cdot \nabla\omega , D \rho_t \>$ exists in the distributional sense, and the gradient estimate below holds:
  $$\sum_{k\in \Z^2_0} \int_0^T \!\! \int \<\sigma_k \cdot \nabla\omega , D \rho_t \>^2 \, \d\mu\d t \leq \|\rho_0 \|_{L^2(\mu)}^2.$$

We repeat the proof of \cite[Theorem 3.2]{FL} to see the difference. Recall that, by convention,
  $$\big\<\sigma_k \cdot \nabla(\Pi_N \omega), D \rho^N_s \big\>_{L^2(\T^2)}^2 \equiv 0, \quad k\notin \Gamma_N.$$
By Proposition \ref{prop-convergence}, there exists a subsequence $\{N_i\}_{i\in \N}$ such that
\begin{itemize}
\item[\rm (a)] $\rho^{N_i}$ converges weakly-$\ast$ to some $\rho$ in  $L^\infty \big(0,T; L^2(H^{-1-}, \mu) \big)$;
\item[\rm (b)] $\tilde\eps_{N_i} \big\<\sigma_k \cdot \nabla(\Pi_{N_i} \omega), D \rho^{N_i}_t \big\>_{L^2(\T^2)}$ converges weakly to some $\varphi \in L^2\big(\Z^2_0 \times [0,T] \times H^{-1-}, \# \otimes \d t \otimes \mu\big)$.
\end{itemize}

Let $\alpha\in C([0,T], \R)$ and $\beta \in L^2\big(\Z^2_0 \times H^{-1-}, \# \otimes \mu\big)$ such that $\beta_k\in \mathcal{FC}_b^2$ for all $k\in \Z^2_0$.  By the assertion (b),
  $$\lim_{i\to \infty} \sum_{k\in \Z^2_0} \int_0^T \!\! \int \tilde\eps_{N_i} \big\<\sigma_k \cdot \nabla(\Pi_{N_i} \omega), D \rho^{N_i}_t \big\>_{L^2(\T^2)} \alpha(t) \beta_k \,\d\mu \d t = \sum_{k\in \Z^2_0} \int_0^T \!\! \int \varphi_k (t)\alpha(t) \beta_k(t) \,\d\mu \d t.$$
Fix some $k\in \Z^2_0$, we assume that $\beta_j \equiv 0 $ for all $j\neq k$ and $\beta_k  = \beta_k\circ \Pi_\Lambda$ for some $\Lambda \Subset \Z_0^2$. Then the above limit reduces to
  \begin{equation}\label{thm-1.1}
  \lim_{i\to \infty} \tilde\eps_{N_i} \int_0^T \!\! \int \big\<\sigma_k \cdot \nabla(\Pi_{N_i} \omega), D \rho^{N_i}_t \big\>_{L^2(\T^2)} \, \alpha(t) \beta_k \,\d\mu \d t = \int_0^T \!\! \int \varphi_k (t)\alpha(t) \beta_k \,\d\mu \d t.
  \end{equation}
For $N_i$ big enough, we have
  $$\aligned
  &\hskip13pt \int_0^T \!\! \int \big\<\sigma_k \cdot \nabla(\Pi_{N_i} \omega), D \rho^{N_i}_t \big\>_{L^2(\T^2)} \, \alpha(t) \beta_k(\omega) \,\d\mu \d t\\
  &= \int_0^T \!\! \int_{H_{N_i}} \big\< G_{N_i}^k, \nabla_{N_i} \rho^{N_i}_t \big\>_{H_{N_i}} \! (\xi) \, \alpha(t) \beta_k(\xi) \,\d\mu_{N_i} \d t \\
  & = - \int_0^T \!\! \int_{H_{N_i}} \rho^{N_i}_t(\xi) \alpha(t) \big\< G_{N_i}^k, \nabla_{N_i} \beta_k \big\>_{H_{N_i}} \!(\xi) \,\d\mu_{N_i} \d t \\
  &= - \int_0^T \!\! \int \rho^{N_i}_t(\omega) \alpha(t) \big\< \sigma_k \cdot \nabla(\Pi_{N_i} \omega), D \beta_k \big\>_{L^2(\T^2)} \,\d\mu \d t.
  \endaligned$$
Therefore, by \cite[Lemma 3.3]{FL},
  $$\aligned
  \tilde\eps_{N_i} \int_0^T \!\! \int \big\<\sigma_k \cdot \nabla(\Pi_{N_i} \omega), D \rho^{N_i}_t \big\>_{L^2(\T^2)} \, \alpha(t) \beta_k \,\d\mu \d t &= - \tilde\eps_{N_i} \int_0^T \!\! \int \rho^{N_i}_t \alpha(t)  \big\< \sigma_k \cdot \nabla\omega, D \beta_k \big\> \,\d\mu \d t \\
  & \to - 0 \cdot \int_0^T \!\! \int \rho_t\, \alpha(t) \big\< \sigma_k \cdot \nabla\omega, D \beta_k \big\> \,\d\mu \d t\\
  &= 0,
  \endaligned$$
where the second step is due to (a). Combining this limit with \eqref{thm-1.1} yields
  $$\int_0^T \!\! \int \varphi_k(t) \alpha(t) \beta_k \,\d\mu \d t = 0.$$
By the arbitrariness of $\alpha\in C([0,T])$ and $\beta_k \in \mathcal{FC}_b^2$, we see that
  \begin{equation*}
  \varphi_k (t) = 0
  \end{equation*}
for all $k\in \Z_0^2$. }
\end{remark}

Now we are ready to present

\begin{proof}[Proof of Theorem \ref{main-result}]
Recall that $\mu_N$ is the standard Gaussian measure on $H_N$. Let $F \in \mathcal{FC}_b^2$ and $\alpha \in C^1 \big([0,T],\R \big)$ satisfying $\alpha(T)=0$. Multiplying both sides of \eqref{Kolmogorov-eq} by $\alpha(t)F$ and integrating by parts with respect to $\mu_N$, we obtain
  $$0= \alpha(0)\int_{H_N} F \rho^{N}_0 \,\d\mu_N + \int_0^T \!\! \int_{H_N} \rho^{N}_s \big[ \alpha'(s) F + \alpha(s) \tilde \L_N F \big] \,\d\mu_N\d s. $$
We transform the integrals to those on $H^{-1-}(\T^2)$ and obtain
  \begin{equation}\label{proof-1}
  \aligned
  0&= \alpha(0)\int F \rho^{N}_0 \,\d\mu + \int_0^T \!\! \int \rho^{N}_s \big[ \alpha'(s) F - \alpha(s) \big\< u(\Pi_N \omega) \cdot \nabla(\Pi_N \omega), D F \big\> \big] \,\d\mu\d s \\
  &\hskip13pt + 4\nu \tilde\eps_N^2 \sum_{k\in \Gamma_N} \int_0^T \!\! \int \rho^{N}_s\, \alpha(s)\, \Big\<\sigma_k \cdot \nabla(\Pi_N \omega), D \big\<\sigma_k \cdot \nabla(\Pi_N \omega), D F \big\> \Big\> \,\d\mu\d s.
  \endaligned
  \end{equation}
  
Assume $F$ has the form $f\circ \Pi_\Lambda$; in this case we say that $F$ is measurable with respect to $H_\Lambda =\mbox{span}\{e_k: k\in \Lambda\}$, or $H_\Lambda$-measurable. Of course, $F$ is also $H_{\Lambda'}$-measurable for any $\Lambda' \supset \Lambda$. When $N$ is big enough, we have $\Lambda \subset \Gamma_N= \Lambda_{N/3}$. For all $k\in \Gamma_{N}$,
  $$\big\< \sigma_k \cdot \nabla(\Pi_{N} \omega), D F\big\> = - \big\< \Pi_{N} \omega, \sigma_k \cdot \nabla (D F) \big\> = - \big\< \omega, \sigma_k \cdot \nabla (D F) \big\> = \< \sigma_k \cdot \nabla \omega, D F \>.$$
We see that $\< \sigma_k \cdot \nabla \omega, D F \>$ is $H_{\Lambda_{2N/3}}$-measurable. In the same way, we have
  $$\aligned
  \Big\< \sigma_k \cdot \nabla(\Pi_{N} \omega), D \big\< \sigma_k \cdot \nabla(\Pi_{N} \omega), D F\big\> \Big\> &= \Big\< \sigma_k \cdot \nabla(\Pi_{N} \omega), D \< \sigma_k \cdot \nabla \omega, D F\> \Big\> \\
  &= \big\< \sigma_k \cdot \nabla \omega, D\< \sigma_k \cdot \nabla\omega, D F\> \big\>,
  \endaligned$$
which is $H_{\Lambda_{N}}$-measurable. Therefore, by \eqref{diffusion-part},
  $$\frac12 \sum_{k\in \Gamma_{N}} \Big\< \sigma_k \cdot \nabla(\Pi_{N} \omega), D \big\< \sigma_k \cdot \nabla(\Pi_{N} \omega), D F\big\> \Big\> = \mathcal{L}_{N}^0 F(\omega),$$
and \eqref{proof-1} becomes
  \begin{equation*}
  \aligned
  0&= \alpha(0)\int F \rho^{N}_0 \,\d\mu + \int_0^T \!\! \int \rho^{N}_s \big[ \alpha'(s) F - \alpha(s) \big\< u(\Pi_N \omega) \cdot \nabla(\Pi_N \omega), D F \big\> \big] \,\d\mu\d s \\
  &\hskip13pt + 8\nu \tilde\eps_N^2 \sum_{k\in \Gamma_N} \int_0^T \!\! \int \rho^{N}_s\, \alpha(s)\, \mathcal{L}_{N}^0 F(\omega) \,\d\mu\d s.
  \endaligned
  \end{equation*}
By Proposition \ref{prop-operator}, changing $N$ into $N_i$ and letting $i\to \infty$, we arrive at
  \begin{equation*}
  \aligned
  0= &\ \alpha(0)\int F \rho_0 \,\d\mu + \int_0^T \!\! \int \rho_s \big[ \alpha'(s) F - \alpha(s) \< u(\omega) \cdot \omega, D F \> \big] \,\d\mu\d s  \\
  & + 4 \nu \pi^2 \sum_{l\in \Lambda} |l|^2 \int_0^T \!\! \int \alpha(s)\, \rho_s\, \big[f_{l,l}(\omega) - f_l(\omega) \< \omega, e_l \>\big] \,\d\mu\d s.
  \endaligned
  \end{equation*}
The proof is complete.
\end{proof}

\section{Appendices}

\subsection{Decomposition of the diffusion part \eqref{prop-operator.1}}

For the reader's convenience, we recall some useful results which were proved in \cite{FL}. First, recall that $C_{k,l}$ is defined in \eqref{coefficients} and $\Lambda_N =\{k\in \Z^2_0: |k|\leq N\}$. The following identity is taken from \cite[Lemma 3.4]{FL}.

\begin{lemma}\label{lem-2-1}
It holds that
  \begin{equation}\label{lem-2-1.1}
  \sum_{k\in \Lambda_N} C_{k,l}^2 = \frac12 \eps_N^{-2} |l|^2 \quad \mbox{with} \quad \eps_N = \bigg(\sum_{k\in \Lambda_N} \frac1{|k|^2} \bigg)^{-1/2}.
  \end{equation}
\end{lemma}

\begin{proof}
Denoting by $D_{k,l}= \frac{k\cdot l}{|k|^2}$, then
  $$C_{k,l}^{2} + D_{k,l}^{2} = \frac{(k^\perp \cdot l)^2}{|k|^4} + \frac{(k \cdot l)^2}{|k|^4}= \frac1{|k|^2} \bigg[ \bigg( \frac{k^\perp}{|k|} \cdot l\bigg)^2 + \bigg( \frac{k}{|k|} \cdot l\bigg)^2 \bigg]= \frac{|l|^2}{|k|^2}.$$
The transformation $k\to k^\perp$ is 1-1 on the set $\Lambda_N =\{k\in \Z^2_0: |k|\leq N\}$, and preserves the norm $|\cdot|$. As a result,
  $$\sum_{k\in \Lambda_N} C_{k,l}^{2} = \sum_{k\in \Lambda_N} \frac{(k^\perp \cdot l)^2}{|k|^4} = \sum_{k\in \Lambda_N} \frac{((k^\perp)^\perp \cdot l)^2}{|k^\perp|^4}= \sum_{k\in \Lambda_N} \frac{(k \cdot l)^2}{|k|^4} = \sum_{k\in \Lambda_N} D_{k,l}^{2}.$$
Combining the above two equalities, we obtain
  \[  \sum_{k\in \Lambda_N} C_{k,l}^{2} = \frac12 \sum_{k\in \Lambda_N} \big( C_{k,l}^{2} + D_{k,l}^{2} \big)= \frac12 |l|^2 \sum_{k\in \Lambda_N} \frac1{|k|^2} = \frac12 \eps_N^{-2}  |l|^2. \qedhere \]
\end{proof}

Next, we recall a decomposition formula of the operator
  $$\L_N^{0} F(\omega) = \frac12 \sum_{k\in \Lambda_N} \big\< \sigma_k\cdot \nabla\omega, D \< \sigma_k\cdot \nabla\omega, D F\> \big\>, \quad F\in \mathcal{FC}_b^2,$$
which was proved in \cite[Proposition 4.2]{FL}. To this end, we need the following simple result.

\begin{lemma}\label{lem-appendix-0}
Assume that $F= f\circ\Pi_\Lambda$ for some finite set $\Lambda\subset \Z_0^2$. We have
  \begin{equation}\label{lem-appendix-0.1}
  \aligned
  \L_N^{0} F(\omega) &= \pi^2 \sum_{k\in \Lambda_N} \sum_{l, m\in \Lambda} C_{k,l} C_{k,m}\, f_{l,m}(\omega) \<\omega, e_k e_{-l}\> \<\omega, e_k e_{-m}\> \\
  &\hskip11pt -\pi^2 \sum_{k\in \Lambda_N} \sum_{l\in \Lambda} C_{k,l}^2 \, f_l(\omega) \big\< \omega, e_k^2 e_l \big\> .
  \endaligned
  \end{equation}
\end{lemma}

\begin{proof}
Note that $DF(\omega )= \sum_{l\in \Lambda} (\partial_l f)(\Pi_\Lambda \omega) e_l = \sum_{l\in \Lambda} f_l(\omega) e_l$; therefore,
  $$\aligned
  \<\sigma_k \cdot \nabla\omega, D F\> &= \sum_{l\in \Lambda} f_l(\omega) \<\sigma_k \cdot \nabla\omega, e_l \> = - \sum_{l\in \Lambda} f_l(\omega) \<\omega, \sigma_k \cdot \nabla e_l \> \\
  &= -\sqrt{2} \pi \sum_{l\in \Lambda} C_{k,l}\, f_l(\omega) \<\omega, e_k e_{-l}\>.
  \endaligned$$
Furthermore,
  $$\aligned
  D \<\sigma_k \cdot \nabla\omega, D F\> &= -\sqrt{2}\pi \sum_{l\in \Lambda} C_{k,l} \big( \<\omega, e_k e_{-l}\> D [ f_l(\omega) ] + f_l(\omega) e_k e_{-l} \big) \\
  &= -\sqrt{2}\pi \sum_{l, m\in \Lambda} C_{k,l} \<\omega, e_k e_{-l}\> f_{l,m}(\omega) e_m  - \sqrt{2}\pi \sum_{l\in \Lambda} C_{k,l}\, f_l(\omega) e_k e_{-l}.
  \endaligned$$
As a result,
  \begin{equation}\label{lem-appendix-0.2}
  \aligned
  \big\< \sigma_k\cdot \nabla\omega, D \< \sigma_k\cdot \nabla\omega, D F\> \big\> & = -\sqrt{2}\pi \sum_{l, m\in \Lambda} C_{k,l}\, f_{l,m}(\omega) \<\omega, e_k e_{-l}\> \< \sigma_k\cdot \nabla\omega, e_m \> \\
  &\hskip11pt - \sqrt{2}\pi \sum_{l\in \Lambda} C_{k,l}\, f_l(\omega) \< \sigma_k\cdot \nabla\omega, e_k e_{-l} \>.
  \endaligned
  \end{equation}
We have $\< \sigma_k\cdot \nabla\omega, e_m \> = -\< \omega, \sigma_k\cdot \nabla e_m\> = - \sqrt{2}\pi C_{k,m} \<\omega, e_k e_{-m}\>$ and
  $$\< \sigma_k\cdot \nabla\omega, e_k e_{-l} \> = - \<\omega, \sigma_k\cdot \nabla(e_k e_{-l}) \> = - \sqrt{2}\pi C_{k,-l} \big\<\omega, e_k^2 e_l \big\> = \sqrt{2}\pi C_{k,l} \big\<\omega, e_k^2 e_l \big\>.$$
Substituting these facts into \eqref{lem-appendix-0.2} and summing over $k$ yield the desired result.
\end{proof}

Now we can rewrite $\L_N^{0} F(\omega)$ as the sum of two parts, in which one part is convergent while the other is in general divergent.

\begin{proposition}\label{prop-operator}
It holds that
  \begin{equation}\label{prop-operator.1}
  \aligned
  \L_N^{0} F(\omega) &= \pi^2 \sum_{l, m\in \Lambda} f_{l,m}(\omega) \sum_{k\in \Lambda_N} C_{k,l} C_{k,m} \big( \<\omega, e_k e_{-l}\> \<\omega, e_k e_{-m}\> - \delta_{l,m} \big) \\
  &\hskip11pt +\frac12 \pi^2 \eps_N^{-2} \sum_{l\in \Lambda} |l|^2 \big[f_{l,l}(\omega) - f_l(\omega) \< \omega, e_l \>\big].
  \endaligned
  \end{equation}
Moreover, for any $l,m\in \Z_0^2$, the quantity
  \begin{equation}\label{prop-operator.2}
  R_{l,m}(N) =\sum_{k\in \Lambda_N} C_{k,l} C_{k,m} \big( \<\omega, e_k e_{-l}\> \<\omega, e_k e_{-m}\> - \delta_{l,m} \big)
  \end{equation}
is a Cauchy sequence in $L^p(H^{-1-}, \mu)$ for any $p>1$.
\end{proposition}

\begin{proof}
The proof of the second assertion is quite long and can be found in the appendix of \cite{FL}. Here we only prove the equality \eqref{prop-operator.1}. We have, by Lemma \ref{lem-2-1},
  \begin{equation}\label{cor-3}
  \aligned
  & \sum_{k\in \Lambda_N} \sum_{l, m\in \Lambda} C_{k,l} C_{k,m}\, f_{l,m}(\omega) \<\omega, e_k e_{-l}\> \<\omega, e_k e_{-m}\> \\
  =& \sum_{l,m\in \Lambda} f_{l,m}(\omega) \sum_{k\in \Lambda_N} C_{k,l} C_{k,m} \big( \< \omega, e_{k} e_{-l} \> \< \omega, e_{k} e_{-m} \> - \delta_{l,m} \big) + \frac12 \eps_N^{-2} \sum_{l\in \Lambda} |l|^2 f_{l,l}(\omega).
  \endaligned
  \end{equation}

Next, note that $C_{-k,l} = -C_{k,l}$ and $e_{k}^2 + e_{-k}^2 \equiv 2$ for all $k\in \Z^2_0$, we have
  \begin{equation*}
  \aligned \sum_{k\in \Lambda_N} C_{k,l}^{2} \big\langle \omega, e_{k}^2 e_{l} \big\rangle &=  \sum_{k\in \Lambda_N, k\in \Z^2_+} \big[ C_{k,l}^{2} \big\langle \omega, e_{k}^2 e_{l} \big\rangle + C_{-k,l}^{2} \big\langle \omega, e_{-k}^2 e_{l} \big\rangle \big] \\
  &= \sum_{k\in \Lambda_N, k\in \Z^2_+} 2 C_{k,l}^{2} \<\omega, e_l\> = \frac12 \eps_N^{-2} |l|^2 \<\omega, e_l\> ,
  \endaligned
  \end{equation*}
where the last step is due to Lemma \ref{lem-2-1}. Therefore,
  $$\sum_{k\in \Lambda_N} \sum_{l\in\Lambda} C_{k,l}^{2}f_{l}(\omega) \big\< \omega, e_{k}^2 e_{l} \big\> = \frac12 \eps_N^{-2} \sum_{l\in\Lambda} |l|^2 f_{l}(\omega) \<\omega, e_l\>. $$
Combining this equality with \eqref{lem-appendix-0.1} and \eqref{cor-3} leads to the identity \eqref{prop-operator.1}.
\end{proof}

\subsection{Coincidence of nonlinear parts}\label{sec-coincidence-nonlinear-part}

Our purpose in this part is to show that the nonlinear term in the vorticity form of the Euler equation defined in \cite[Theorem 8]{F1} agrees with that defined by Galerkin approximation; therefore, we can freely use any of them. Let $\{\tilde e_k\}_{k\in \Z^2}$ be the canonical complex orthonormal basis of $L^2\big(\T^2, \mathbb C\big)$; then $\{\tilde e_k\otimes \tilde e_l \}_{k,l\in \Z^2}$ is an orthonomal basis of $L^2\big((\T^2)^2, \mathbb C\big)$.

\begin{lemma}\label{lem-appendix}
Assume $f\in C^\infty \big((\T^2)^2,\R \big) $ is symmetric and $\int_{\T^2} f(x,x)\,\d x=0$. Then
  $$\<\omega\otimes \omega, f\> = \sum_{k,l\in \Z^2} f_{k,l} \<\omega, \tilde e_k\> \<\omega, \tilde e_l\> \quad \mbox{holds in } L^2\big(H^{-1-}, \mu \big),$$
where
  $$f_{k,l}= \<f, \tilde e_k\otimes \tilde e_l\>= \int_{(\T^2)^2} f(x,y)\tilde e_{k}(x) \tilde e_{l}(y)\,\d x\d y.$$
\end{lemma}

\begin{proof}
Denote by
  \begin{equation}\label{lem-appendix.0}
  \hat\Lambda_N = \{k\in \Z^2: |k|\leq N\} = \Lambda_N\cup \{0\}.
  \end{equation}
Since $f\in C^\infty((\T^2)^2)$, the partial sum of the Fourier series
  $$f_N(x,y):= \sum_{k,l\in \hat\Lambda_N} f_{k,l}\, \tilde e_k(x) \tilde e_l(y) $$
converges to $f$, uniformly on $(\T^2)^2$ and in $L^2((\T^2)^2)$. In particular,
  \begin{equation}\label{lem-appendix.1}
  \lim_{N\to\infty} \int_{\T^2} f_N(x,x)\,\d x = \int_{\T^2} f(x,x)\,\d x =0.
  \end{equation}
It is obvious that $f_N(x,y)$ is smooth and symmetric. By \cite[Corollary 6, ii), iii)]{F1},
  \begin{equation*}
  \E_\mu \bigg[ \Big(\<\omega\otimes \omega, f- f_N\> + \int_{\T^2} f_N(x,x)\,\d x \Big)^2 \bigg] = 2 \int_{(\T^2)^2} (f- f_N)^2(x,y) \,\d x\d y.
  \end{equation*}
As a result,
  \begin{equation}\label{lem-appendix.2}
  \E_\mu \big[\<\omega\otimes \omega, f- f_N\>^2 \big] \leq 4\int_{(\T^2)^2} (f- f_N)^2(x,y) \,\d x\d y + 2 \bigg(\int_{\T^2} f_N(x,x)\,\d x \bigg)^2.
  \end{equation}

Next, note that
  $$\<\omega\otimes \omega, f_N\>= \sum_{k,l\in \hat\Lambda_N} f_{k,l} \<\omega, \tilde e_k\>\<\omega, \tilde e_l\>.$$
Therefore, by \eqref{lem-appendix.2},
  $$\aligned
  & \E_\mu \bigg[\Big(\<\omega\otimes \omega, f\> - \sum_{k,l\in \hat\Lambda_N} f_{k,l} \<\omega, \tilde e_k\>\<\omega, \tilde e_l\>\Big)^2 \bigg] \\
  \leq & \, 4\int_{(\T^2)^2} (f- f_N)^2(x,y) \,\d x\d y + 2 \bigg(\int_{\T^2} f_N(x,x)\,\d x \bigg)^2.
  \endaligned$$
Thanks to \eqref{lem-appendix.1}, the desired result follows by letting $N\to \infty$.
\end{proof}

We need the following simple equality.

\begin{lemma}\label{lem-appendix-2}
Let $\{a_{k,l} \}_{k,l\in \hat\Lambda_N} \subset \mathbb C$ be satisfying $a_{k,l}= a_{l,k}$, $\overline{a_{k,l}}= a_{-k,-l}$. Then
  \begin{equation}\label{lem-appendix-2.1}
  \E_\mu \bigg[ \bigg|\sum_{k,l\in \hat\Lambda_N} a_{k,l} \<\omega, \tilde e_k\>\<\omega, \tilde e_l\> - \sum_{k\in \hat\Lambda_N} a_{k,-k} \bigg|^2\bigg] = 2\sum_{k,l\in \hat\Lambda_N} |a_{k,l}|^2.
  \end{equation}
\end{lemma}

\begin{proof}
It is clear that $\sum_{k,l\in \hat\Lambda_N} a_{k,l} \<\omega, \tilde e_k\>\<\omega, \tilde e_l\>$ is real and
  $$\sum_{k\in \hat\Lambda_N} a_{k,-k} = \E_\mu \bigg(\sum_{k,l\in \hat\Lambda_N} a_{k,l} \<\omega, \tilde e_k\>\<\omega, \tilde e_l\>\bigg).$$
Following the arguments of \cite[Lemma 5.1]{FL}, we can prove the desired equality.
\end{proof}

Recall the expression of $H_\phi$ for $\phi\in C^\infty(\T^2)$ in Remark \ref{sec-2-remark}. Now we can prove the intermediate result below.

\begin{proposition}\label{prop-appendix}
For any $j\in \Z_0^2$, the following identity holds in $L^2\big(H^{-1-}, \mu \big)$:
  $$\<\omega\otimes \omega, H_{e_j}\> = \sum_{k,l\in \Z^2} \<H_{e_j}, \tilde e_k\otimes \tilde e_l\> \<\omega, \tilde e_k\> \<\omega, \tilde e_l\> ,$$
where $e_j$ is defined in \eqref{ONB}.
\end{proposition}

\begin{proof}
Let $H^n_{e_j}$ be the functions constructed in \cite[Remark 9]{F1}, which satisfy the conditions in Lemma \ref{lem-appendix}. Recall the definition of $\hat\Lambda_N$ in \eqref{lem-appendix.0}. To simplify the notations, we introduce
  $$\hat \omega_N = \hat \Pi_N \omega = \sum_{k\in \hat\Lambda_N} \<\omega, \tilde e_k\> \tilde e_k,\quad \omega\in H^{-1-}.$$
Then
  $$\<\hat \omega_N\otimes \hat \omega_N, H_{e_j}\> = \sum_{k,l\in \hat\Lambda_N} \<H_{e_j}, \tilde e_k\otimes \tilde e_l\> \<\omega, \tilde e_k\>\<\omega, \tilde e_l\> $$
is the partial sum of the series. We have
  \begin{equation}\label{prop-appendix.1}
  \aligned
  & \E_\mu \big[\big(\<\omega\otimes \omega, H_{e_j}\> - \<\hat \omega_N\otimes \hat \omega_N, H_{e_j}\>\big)^2 \big] \\
  \leq & \, 3\, \E_\mu \big[\<\omega\otimes \omega, H_{e_j} - H^n_{e_j}\>^2 \big] + 3\, \E\big[\big(\<\omega\otimes \omega, H^n_{e_j}\> - \<\hat \omega_N\otimes \hat \omega_N, H^n_{e_j}\> \big)^2 \big] \\
  & + 3\, \E_\mu \big[\<\hat \omega_N\otimes \hat \omega_N, H^n_{e_j}- H_{e_j}\>^2 \big].
  \endaligned
  \end{equation}

We estimate the three terms one-by-one. By the proof of \cite[Theorem 8]{F1},
  \begin{equation}\label{prop-appendix.2}
  \E_\mu \big[\<\omega\otimes \omega, H_{e_j} - H^n_{e_j}\>^2 \big] \leq 2 \int_{(\T^2)^2} \big(H_{e_j} - H^n_{e_j} \big)^2(x,y) \,\d x\d y.
  \end{equation}
Next, by Lemmas \ref{lem-app-3} and \ref{lem-app-4} below (see also \cite[Lemma 5.3]{FL}), we have $\E_\mu \<\hat \omega_N\otimes \hat \omega_N, H_{e_j}\>=0$. Moreover, for any fixed $n\geq 1$, Lemma \ref{lem-appendix} implies
  \begin{equation}\label{prop-appendix.2.5}
  \E_\mu \big[ \big(\<\hat \omega_N\otimes \hat \omega_N, H^n_{e_j}\> - \<\omega\otimes \omega, H^n_{e_j}\>\big)^2\big]\to 0 \quad \mbox{as } N\to \infty.
  \end{equation}

It remains to deal with the last term on the r.h.s. of \eqref{prop-appendix.1}. As a result of \eqref{prop-appendix.2.5},
  \begin{equation}\label{prop-appendix.3}
  \lim_{N\to \infty} \E_\mu \<\hat \omega_N\otimes \hat \omega_N, H^n_{e_j}\> =  \E_\mu \<\omega\otimes \omega, H^n_{e_j}\> = \int_{\T^2} H^n_{e_j}(x,x)\,\d x=0,
  \end{equation}
where the second step is due to \cite[Corollary 6, ii)]{F1}. By \eqref{lem-appendix-2.1},
  $$\aligned
  &\, \E_\mu \big[\big(\<\hat \omega_N\otimes \hat \omega_N, H^n_{e_j}- H_{e_j}\> - \E_\mu \<\hat \omega_N\otimes \hat \omega_N, H^n_{e_j}- H_{e_j}\> \big)^2 \big]\\
  = &\, 2 \sum_{k,l\in \hat\Lambda_N} \big| \big\<H^n_{e_j}- H_{e_j}, \tilde e_k\otimes \tilde e_l \big\> \big|^2 \leq 2 \int_{(\T^2)^2} \big(H^n_{e_j}- H_{e_j}\big)^2(x,y)\,\d x\d y .
  \endaligned $$
Therefore,
  $$\aligned
  \E_\mu \big[\<\hat \omega_N\otimes \hat \omega_N, H^n_{e_j}- H_{e_j}\>^2 \big]\leq &\, 4 \int_{(\T^2)^2} \big(H^n_{e_j}- H_{e_j}\big)^2(x,y)\,\d x\d y + 2 \big[ \E_\mu \<\hat \omega_N\otimes \hat \omega_N, H^n_{e_j}- H_{e_j}\> \big]^2 \\
  = &\, 4 \int_{(\T^2)^2} \big(H^n_{e_j}- H_{e_j}\big)^2(x,y)\,\d x\d y + 2 \big[ \E_\mu \<\hat \omega_N\otimes \hat \omega_N, H^n_{e_j} \> \big]^2,
  \endaligned $$
where we used again $\E_\mu \<\hat \omega_N\otimes \hat \omega_N, H_{e_j}\>= 0$. Thanks to \eqref{prop-appendix.3},
  $$\limsup_{N\to\infty} \E\big[\<\hat \omega_N\otimes \hat \omega_N, H^n_{e_j}- H_{e_j}\>^2 \big]\leq 4 \int_{(\T^2)^2} \big(H^n_{e_j}- H_{e_j}\big)^2(x,y)\,\d x\d y.$$
Combining the above inequality with \eqref{prop-appendix.1}--\eqref{prop-appendix.2.5}, first letting $N\to \infty$ in \eqref{prop-appendix.1} yield
  $$\limsup_{N\to\infty} \E\big[\big(\<\omega\otimes \omega, H_{e_j}\> - \<\hat \omega_N\otimes \hat \omega_N, H_{e_j}\>\big)^2 \big] \leq 18 \int_{(\T^2)^2} \big(H^n_{e_j}- H_{e_j}\big)^2(x,y)\,\d x\d y.$$
We finish the proof by sending $n\to \infty$.
\end{proof}

Recall that we have defined the projection
  $$\omega_N = \Pi_N \omega = \sum_{k\in \Lambda_N} \<\omega, \tilde e_k\> \tilde e_k.$$
According to \eqref{lem-appendix.0}, we have $\hat \omega_N =\omega_N + \<\omega, 1\>$. Taking into account Lemma \ref{lem-app-3} below, we conclude that, for any $j \in \Z_0^2$,
  \begin{equation}\label{appendix.0}
  \<\hat \omega_N\otimes \hat \omega_N, H_{e_j}\> = \<\omega_N\otimes \omega_N, H_{e_j}\>\quad \mbox{for all } N\geq 1.
  \end{equation}
It remains to prove

\begin{lemma}\label{lem-app-3}
For any $j\in \Z_0^2$,
  $$\<H_{e_j}, \tilde e_k\otimes \tilde e_l\> =0 \quad \mbox{for } k=0 \mbox{ or } l=0.$$
\end{lemma}

\begin{proof}
We have
  \begin{equation}\label{lem-app-3.0}
  H_{e_j}(x,y) = \pi (e_{-j}(x) - e_{-j}(y))\, j \cdot K(x-y),\quad (x,y)\in \T^2 \times \T^2.
  \end{equation}
Without loss of generality, we assume $j\in \Z^2_+$ thus $-j\in \Z^2_-$ and
  \begin{equation}\label{lem-app-3.1}
  e_{-j}(x) = \frac{1}{\sqrt{2}\, {\rm i}} \big[\tilde e_{-j}(x) - \tilde e_j(x)\big].
  \end{equation}
Recall that
  \begin{equation}\label{lem-app-3.2}
  \tilde e_k\ast K = 2\pi {\rm i}\, \delta_{k\neq 0} \frac{k^\perp}{|k|^2} \tilde e_k\quad \mbox{for all } k\in \Z^2.
  \end{equation}

\emph{Case 1:} $k=l=0$. We have
  $$\int_{(\T^2)^2} H_{e_j}(x,y)\,\d x\d y= \pi j \cdot \int_{(\T^2)^2} (e_{-j}(x) - e_{-j}(y)) K(x-y)\,\d x\d y = -2 \pi j \cdot \int_{\T^2} (e_{-j}\ast K)(x)\,\d x. $$
Using \eqref{lem-app-3.1} and \eqref{lem-app-3.2}, we obtain
  $$\int_{(\T^2)^2} H_{e_j}(x,y)\,\d x\d y= -2 \pi j \cdot \int_{\T^2} \frac{1}{\sqrt{2}\, {\rm i}} \bigg(2\pi {\rm i} \frac{(-j)^\perp}{|j|^2}\tilde e_{-j}(x) - 2\pi {\rm i} \frac{j^\perp}{|j|^2}\tilde e_j(x) \bigg) \d x =0.$$

\emph{Case 2:} $k=0$ and $l\neq 0$. Then
  $$\<H_{e_j}, \tilde e_0\otimes \tilde e_l\> = \int_{(\T^2)^2} H_{e_j}(x,y) \tilde e_l(y)\,\d x\d y = \pi j \cdot  \int_{(\T^2)^2} (e_{-j}(x)- e_{-j}(y)) K(x-y)\tilde e_l(y)\,\d x\d y. $$
We divide the r.h.s. into two terms $I_1$ and $I_2$. We have, by \eqref{lem-app-3.2},
  $$I_1= \pi j \cdot  \int_{\T^2} e_{-j}(x) (K\ast \tilde e_l)(x)\,\d x= 2\pi^2 {\rm i} \frac{j\cdot l^\perp}{|l|^2} \int_{\T^2} e_{-j}(x) \tilde e_l(x)\,\d x.$$
According to \eqref{lem-app-3.1}, it is clear that if $l\neq \pm j$, then $I_1=0$. On the other hand, if $l= j$ or $l=-j$, we still have $I_1=0$.

Next, we deal with $I_2$. Again by \eqref{lem-app-3.1},
  \begin{equation}\label{lem-app-3.3}
  \aligned
  I_2 &= -\frac{\pi}{\sqrt{2}\, {\rm i}}j\cdot \int_{(\T^2)^2} \big[\tilde e_{-j}(y) - \tilde e_j(y)\big] K(x-y) \tilde e_l(y)\,\d x\d y\\
  &= -\frac{\pi}{\sqrt{2}\, {\rm i}}j\cdot \int_{\T^2} \big[(K\ast \tilde e_{l-j})(x) - (K\ast \tilde e_{l+j})(x)\big]\,\d x.
  \endaligned
  \end{equation}
If $l=j$, then by \eqref{lem-app-3.2},
  $$I_2 = \frac{\pi}{\sqrt{2}\, {\rm i}}j\cdot \int_{\T^2} 2\pi {\rm i} \frac{(2j)^\perp}{|2j|^2} \tilde e_{2j}(x)\,\d x =0.$$
Similarly, $I_2=0$ if $l=-j$. Finally, if $l\neq \pm j$, then we deduce easily from \eqref{lem-app-3.2} and \eqref{lem-app-3.3} that $I_2=0$.

Summarizing these computations, we conclude that $\<H_{e_j}, \tilde e_0\otimes \tilde e_l\>=0$ for all $l\in \Z_0^2$.

\emph{Case 3:} $k\neq 0$ and $l= 0$. The arguments are similar as in the second case and we omit it here. We can also deduce the result by using the symmetry property of $H_{e_j}$.
\end{proof}

Now we can prove the first main result of this part.

\begin{theorem}\label{thm-appendix}
For any $j\in \Z_0^2$,
  $$\<\omega\otimes \omega, H_{e_j}\> = \lim_{N\to \infty} \<\omega_N\otimes \omega_N, H_{e_j}\>\quad \mbox{holds in } L^2 \big(H^{-1-}, \mu \big).$$
Moreover,
  \begin{equation}\label{thm-appendix.1}
  \E \big[ \<\omega\otimes \omega, H_{e_j}\>^2 \big] = 2 \int_{(\T^2)^2} H_{e_j}^2(x,y)\,\d x\d y = 2 \sum_{k,l\in\Z_0^2} |\<H_{e_j}, \tilde e_k\otimes \tilde e_l\>|^2.
  \end{equation}
\end{theorem}

\begin{proof}
The first assertion follows from Proposition \ref{prop-appendix} and \eqref{appendix.0}. Next, by Lemma \ref{lem-app-4} below,
  $$\<H_{e_j}, \tilde e_k\otimes \tilde e_{-k}\> =0 \quad \mbox{for all } k\in\Z^2.$$
Hence, Lemma \ref{lem-appendix-2} and \eqref{appendix.0} imply
  $$ \E \big[ \<\omega_N\otimes \omega_N, H_{e_j}\>^2 \big] = \E \bigg[ \bigg|\sum_{k,l\in \Lambda_N} \<H_{e_j}, \tilde e_k\otimes \tilde e_l\> \<\omega, \tilde e_k\>\<\omega, \tilde e_l\> \bigg|^2\bigg] = 2\sum_{k,l\in \Lambda_N} |\<H_{e_j}, \tilde e_k\otimes \tilde e_l\>|^2.$$
Letting $N\to \infty$ yields the second result.
\end{proof}

In the following, we denote formally by
  \begin{equation}\label{app-definition}
  b(\omega) = u(\omega)\cdot \nabla \omega, \quad b_N(\omega) = \Pi_N\big[ u(\omega_N)\cdot \nabla \omega_N \big].
  \end{equation}
We shall prove that $b$ is well defined as an element in $L^2\big(H^{-1-}(\T^2),\mu; H^{-2-}(\T^2)\big)$ and $b_N\to b$ w.r.t. the norm of this space as $N\to \infty$. This assertion is consistent with \cite[Proposition 3.1]{DaPD} and \cite[Proposition 3.2]{AF}.

For any $j\in \Z_0^2$, by Theorem \ref{thm-appendix},
  \begin{equation}\label{app-expression}
  \<b(\omega), e_j\> = -\<\omega\otimes \omega, H_{e_j}\> = -\sum_{k,l\in \Z_0^2} \<H_{e_j}, \tilde e_k\otimes \tilde e_l\> \<\omega, \tilde e_k\>\<\omega, \tilde e_l\>.
  \end{equation}
We need the following preparation.

\begin{lemma}\label{lem-app-4}
For all $j,k,l\in \Z_0^2$,
  $$\<H_{e_j}, \tilde e_k\otimes \tilde e_l\>= \sqrt{2}\, \pi^2 \bigg(\frac{j \cdot l^\perp}{|l|^2} + \frac{j \cdot k^\perp}{|k|^2} \bigg) \times \begin{cases}
  \delta_{j,k+l} - \delta_{j,-k-l}, & j\in \Z^2_+;\\
  {\rm i}\, (\delta_{j,k+l} + \delta_{j,-k-l}) , & j\in \Z^2_-.
  \end{cases}$$
\end{lemma}

\begin{proof}
Assume $j\in \Z^2_+$. By \eqref{lem-app-3.0},
  $$\<H_{e_j}, \tilde e_k\otimes \tilde e_l\>= \pi j \cdot \int_{\T^2} e_{-j}(x) \tilde e_k(x) (K\ast \tilde e_l)(x)\,\d x + \pi j \cdot \int_{\T^2} e_{-j}(y) \tilde e_l(y) (K\ast \tilde e_k)(y)\,\d y. $$
We denote the two terms by $J_1$ and $J_2$. By \eqref{lem-app-3.2} and \eqref{lem-app-3.1},
  $$J_1= 2\pi^2 {\rm i} \frac{j \cdot l^\perp}{|l|^2} \int_{\T^2} e_{-j}(x) \tilde e_k(x) \tilde e_l(x)\,\d x = \sqrt{2}\, \pi^2 \frac{j \cdot l^\perp}{|l|^2} (\delta_{j,k+l} - \delta_{j,-k-l}). $$
Similarly,
  $$J_2= \sqrt{2}\, \pi^2 \frac{j \cdot k^\perp}{|k|^2} (\delta_{j,k+l} - \delta_{j,-k-l}).$$
The proof is complete.
\end{proof}

Lemma \ref{lem-app-4} implies that
  $$|\<H_{e_j}, \tilde e_k\otimes \tilde e_l\>|^2 = 2\pi^4 \bigg(\frac{j \cdot l^\perp}{|l|^2} + \frac{j \cdot k^\perp}{|k|^2} \bigg)^2 (\delta_{j,k+l} + \delta_{j,-k-l}) = 2\pi^4 \delta_{j,\pm(k+l)} \bigg(\frac{j \cdot l^\perp}{|l|^2} + \frac{j \cdot k^\perp}{|k|^2} \bigg)^2.$$
Using the fact $j= \pm(k+l)$, we obtain
  \begin{equation}\label{app.1}
  |\<H_{e_j}, \tilde e_k\otimes \tilde e_l\>|^2=2\pi^4 \delta_{j,\pm(k+l)} (j\cdot k^\perp)^2 \bigg(\frac1{|l|^2} - \frac1{|k|^2} \bigg)^2.
  \end{equation}
The next estimate will play a key role in the sequel.

\begin{lemma}\label{lem-app-5}
There exists $C>0$ such that for all $|j|\geq 2$,
  $$\sum_{k,l\in \Z_0^2} |\<H_{e_j}, \tilde e_k\otimes \tilde e_l\>|^2 \leq C|j|^2 \log|j|.$$
\end{lemma}

\begin{proof}
Thanks to \eqref{app.1}, we have
  $$\aligned
  \sum_{k,l\in \Z_0^2} |\<H_{e_j}, \tilde e_k\otimes \tilde e_l\>|^2=&\, 2\pi^4 \sum_{k\in \Z_0^2 \setminus\{j \}} (j\cdot k^\perp)^2 \bigg(\frac1{|j-k|^2} - \frac1{|k|^2} \bigg)^2 \\
  & + 2\pi^4 \sum_{k\in \Z_0^2 \setminus\{-j \}} (j\cdot k^\perp)^2 \bigg(\frac1{|j+k|^2} - \frac1{|k|^2} \bigg)^2
  \endaligned$$
which is easily seen to be convergent. We denote the two quantities on the r.h.s. by $I_1$ and $I_2$, respectively. Note that
  $$ \bigg(\frac1{|j-k|^2} - \frac1{|k|^2} \bigg)^2 = \frac{(|j|^2- 2 j\cdot k)^2}{|j-k|^4 |k|^4} \leq 2 \frac{|j|^4 + 4 (j\cdot k)^2}{|j-k|^4 |k|^4},$$
thus
  $$I_1 \leq 4\pi^4|j|^4 \sum_{k\in \Z_0^2 \setminus\{j \}} \frac{ (j\cdot k^\perp)^2 }{|j-k|^4 |k|^4} + 16\pi^4 \sum_{k\in \Z_0^2 \setminus\{j \}} \frac{ (j\cdot k^\perp)^2 (j\cdot k)^2}{|j-k|^4 |k|^4} =: I_{1,1} + I_{1,2}. $$
We have
  $$I_{1,1} = 4\pi^4|j|^4 \sum_{k\in \Z_0^2 \setminus\{j \}} \frac{ ((j-k)\cdot k^\perp)^2 }{|j-k|^4 |k|^4} \leq 4\pi^4|j|^4 \sum_{k\in \Z_0^2 \setminus\{j \}} \frac{ 1}{|j-k|^2 |k|^2} \leq C|j|^2 \log|j|,$$
where the last step is due to \cite[Proposition A.1]{AF}. Similarly,
  $$I_{1,2} = 16\pi^4 \sum_{k\in \Z_0^2 \setminus\{j \}} \frac{ (j\cdot (k-j)^\perp)^2 (j\cdot k)^2}{|j-k|^4 |k|^4} \leq 16\pi^4 |j|^4 \sum_{k\in \Z_0^2 \setminus\{j \}} \frac{ 1}{|j-k|^2 |k|^2} \leq C|j|^2 \log|j|.
  $$
Therefore, we obtain
  \begin{equation}\label{lem-app-5.1}
  I_1 \leq C|j|^2 \log|j|.
  \end{equation}
In the same way, we have $ I_2 \leq C|j|^2 \log|j|$ which, together with \eqref{lem-app-5.1}, implies the result.
\end{proof}

\begin{remark}
Recall the definition of $\mathcal L_\infty$. The above estimate shows that the nonlinear part in $\mathcal L_\infty$ is not dominated by the diffusion part. Indeed, taking $F(\omega)= \<\omega, e_j\>,\, |j|\geq 2$, then by \eqref{app-expression}, Theorem \ref{thm-appendix} and Lemma \ref{lem-app-5},
  $$\E_\mu \big[\<b(\omega), DF\>^2 \big] = \E_\mu \big[\<\omega\otimes \omega, H_{e_j}\>^2 \big] \leq C |j|^2 \log|j|.$$
Note that the factor $\log|j|$ cannot be eliminated. On the other hand, regarding the diffusion part $\mathcal L_\infty^D$ in $\mathcal L_\infty$, we have
  $$ -\E_\mu \big[F \mathcal L_\infty^D F\big] = 4\pi^2 |j|^2.$$
As a result, the Lions approach does not work here to give us the uniqueness of solutions to \eqref{modified-eq}.
\end{remark}

Now we can prove the second main result of this part.

\begin{theorem}\label{thm-app-2}
For any $\delta>0$,
  $$\lim_{N\to \infty} \E_\mu \big(\|b_N(\omega) - b(\omega)\|_{H^{-2-\delta}(\T^2)}^2 \big) =0.$$
\end{theorem}

\begin{proof}
Note that
  $$\|b_N(\omega) - b(\omega)\|_{H^{-2-\delta}(\T^2)}^2= \sum_{j\in\Z_0^2} \frac1{|j|^{4+2\delta}} \big(\<b_N(\omega) , e_j\> - \<b(\omega), e_j\>\big)^2$$
and by \eqref{app-definition},
  $$\<b_N(\omega), e_j\> = - {\bf 1}_{\Lambda_N}(j) \<\omega_N\otimes \omega_N, H_{e_j}\> = - {\bf 1}_{\Lambda_N}(j) \sum_{k,l\in \Lambda_N} \<H_{e_j}, \tilde e_k\otimes \tilde e_l\> \<\omega, \tilde e_k\>\<\omega, \tilde e_l\>.$$
Therefore,
  $$\|b_N(\omega) - b(\omega)\|_{H^{-2-\delta}(\T^2)}^2= \sum_{j\in \Lambda_N} \frac1{|j|^{4+2\delta}} \big(\<b_N(\omega) , e_j\> - \<b(\omega), e_j\>\big)^2 + \sum_{j\in \Lambda_N^c } \frac{\<b(\omega), e_j\>^2}{|j|^{4+2\delta}} .$$
Denote the two quantities by $J_{1,N}$ and $J_{2,N}$ respectively.

First, by Theorem \ref{thm-appendix} and \eqref{app-expression}, we have
  $$\E J_{2,N}= \sum_{j\in \Lambda_N^c } \frac{\E \<b(\omega), e_j\>^2}{|j|^{4+2\delta}} = 2 \sum_{j\in \Lambda_N^c } \frac1{|j|^{4+2\delta}} \sum_{k,l\in \Z_0^2} |\<H_{e_j}, \tilde e_k\otimes \tilde e_l\>|^2 \leq C \sum_{j\in \Lambda_N^c } \frac{\log|j|}{|j|^{2+2\delta}},$$
where the last inequality follows from Lemma \ref{lem-app-5}. Therefore,
  \begin{equation}\label{thm-app-2.1}
  \lim_{N\to \infty} \E J_{2,N} =0.
  \end{equation}

Recalling \eqref{app-expression} and denoting by $\Lambda_{N,N}^c= (\Z_0^2 \times \Z_0^2) \setminus (\Lambda_N\times \Lambda_N)$, we arrive at
  $$\aligned
  \<b_N(\omega) , e_j\> - \<b(\omega), e_j\> = & \sum_{(k,l)\in \Lambda_{N,N}^c} \<H_{e_j}, \tilde e_k\otimes \tilde e_l\> \<\omega, \tilde e_k\>\<\omega, \tilde e_l\>,\quad j\in \Lambda_N .
  \endaligned$$
Analogous to \eqref{thm-appendix.1},
  $$\E \big(\<b_N(\omega) , e_j\> - \<b(\omega), e_j\>\big)^2 = 2 \sum_{(k,l)\in \Lambda_{N,N}^c} |\<H_{e_j}, \tilde e_k\otimes \tilde e_l\>|^2. $$
As a result,
  $$\E J_{1,N} = 2 \sum_{j\in \Lambda_N} \frac1{|j|^{4+2\delta}} \sum_{(k,l)\in \Lambda_{N,N}^c} |\<H_{e_j}, \tilde e_k\otimes \tilde e_l\>|^2 \leq 2 \sum_{j\in \Z_0^2} \frac1{|j|^{4+2\delta}} \sum_{(k,l)\in \Lambda_{N,N}^c} |\<H_{e_j}, \tilde e_k\otimes \tilde e_l\>|^2.$$
By Lemma \ref{lem-app-5} and the dominated convergence theorem, we obtain
  $$\lim_{N\to \infty} \E J_{1,N} =0.$$
Combining this limit with \eqref{thm-app-2.1}, we complete the proof.
\end{proof}

\medskip

\noindent \textbf{Acknowledgements.} The second author is supported by the National Natural Science Foundation of China (Nos. 11571347, 11688101), the Youth Innovation Promotion Association, CAS (2017003) and the Special Talent Program of the Academy of Mathematics and Systems Science, CAS.

\end{document}